\documentclass[12pt]{article}%
\usepackage{amsmath}
\usepackage{amsfonts}
\usepackage{amssymb}
\usepackage{graphicx}%
\providecommand{\U}[1]{\protect\rule{.1in}{.1in}}
\newtheorem{theorem}{Theorem}

\newtheorem{corollary}[theorem]{Corollary}

\newtheorem{definition}[theorem]{Definition}

\newtheorem{lemma}[theorem]{Lemma}

\newtheorem{proposition}[theorem]{Proposition}
\newtheorem{remark}[theorem]{Remark}

\newenvironment{proof}[1][Proof]{\noindent\textbf{#1.} }{\ \rule{0.5em}{0.5em}}
\begin{document}

\title{The sup-completion of a Dedekind complete vector lattice II}
\author{Youssef Azouzi\thanks{The authors are members of the GOSAEF research group}
and Youssef Nasri \thanks{This document is the results of the research project
funded by the National Science Foundation}\\{\small Research Laboratory of Algebra, Topology, Arithmetic, and Order}\\{\small Department of Mathematics}\\COSAEF {\small Faculty of Mathematical, Physical and Natural Sciences of
Tunis}\\{\small Tunis-El Manar University, 2092-El Manar, Tunisia}}
\maketitle

\begin{abstract}
We persist in our investigation of the sup-completion of a Dedekind complete
Riesz space, extending to the broader context of Riesz spaces. some results
initially obtained by Feng, Li, Shen, and also by Erd\"{o}s, and R\'{e}nyi.

\end{abstract}

\section{Introduction}

In this paper, we continue our investigation of the sup-completion of a
Dedekind complete Riesz space started in \cite{L-900}. We delve deeper into
the decomposition of finite and infinite parts, initially introduced in
\cite{L-900}, and further investigate the properties elucidated in that study.
Within our work, we introduce a new concept that we call the `star map' as a
pivotal construct necessary for generalizing results from measure theory or
classical stochastic theory to the domain of Riesz spaces. As we encounter
instances where we seek to apply an inverse operation amidst dealing with
non-invertible elements, we address this issue by introducing the notion of a
'partial inverse'. While briefly discussed in our previous work \cite{L-900},
this concept will be systematically explored here with comprehensive details.
Consider a Dedekind complete Riesz space $X$ with a weak order unit, denoted
by $e.$ Then the universal completion $X^{u}$ of $X$ has a natural structure
of an $f$-algebra, where $e$ serves as the identity element$.$ Each element
$x$ in $X$ functions as a weak unit within the band $B_{x}$ generated by $x$
in $X^{u}.$ Consequently $x$ has an inverse in that band, referred to as the
partial inverse of $x.$ If $x$ is a positive element in the cone $X_{+}^{s},$
where $X^{s}$ denotes the sup-completion of $X,$ we denote by $x^{\ast}$ the
partial inverse of its finite part $x^{f}$. This partial inverse is also
recently used by Roelands and Schwanke in \cite{L-424} and they adopted the
same notation. It is also used in \cite{L-886} to develop a Hahn-Jordan
theorem in Riesz spaces. Our motivation here is to get a Riesz space version
of a result obtained by Feng, Li and Shen in \cite{a-1843}. A weaker form of
this result was obtained earlier by Erd\"{o}s and R\'{e}nyi in \cite{a-1971}
that allows to get a generalization of Borel Cantelli Lemma.

Let us give a brief outline of the content of the paper. Section 2 provides
some preliminaries. Sections 3 and 4 are devoted to present new results
concerning the sup-completion of a Dedekind complete Riesz space. In the first
part we investigate finite and infinite parts. The second part deals with
partial inverses of elements of $X^{s}.$ We introduce that map $x\longmapsto
x^{\ast}$ where $x^{\ast}$ is the inverse of $x^{f}$ in the band $B_{x^{f}}.$
Then we prove under some conditions that if $\left(  x_{\alpha}\right)  $
converges to $x$ in order then $x_{\alpha}^{\ast}$ converges in order to
$x^{\ast}.$ In the last section we apply our results to obtain a
generalization of a theorem of Feng, Li and Shen to the setting of Riesz
spaces. The reader is referred to \cite{b-1665} for the definition of the
sup-completion, a fundamental concept in this paper, and to the papers
\cite{L-444} and \cite{L-900} for more informations of that notion. All
unexplained terminology and notation concerning Riesz spaces can be found in
standard references \cite{b-240}, \cite{b-1087} and \cite{b-1089}.

\section{Preliminaries}

We consider a Dedekind complete Riesz space $X.$ We employ $X^{u}$ to
represent its universal completion, while its sup-completion is denoted by
$X^{s}.$ Recall that $X^{s}$ is a lattice ordered cone that contains $X,$ and
which has a greatest element that we denote by $\infty.$ If $B$ is a band in
$X$ then its sup-completion $B^{s}$ is contained in $X^{s}$ (see \cite[Theorem
6]{L-444}) and its greatest element will be denoted by $\infty_{B}.$ More
about the space $X^{s}$ can be found in \cite{L-444,L-900}. We denote by
$\mathcal{B}\left(  X\right)  $ the Boolean algebra of projection bands in
$X$. To a band $B\in\mathcal{B}\left(  X\right)  $ we associate the band
projection $P_{B}$ on $B$ and we use the notation $P^{d}=I-P$ for any band
projection $P.$ We shorten $P_{B_{x}}$ to $P_{x},$ with $B_{x}$ denoting the
principal band generated by $x.$ It should be noted that this notion can be
extended in a natural manner to elements in $X^{s}.$ It was shown indeed in
\cite[Lemma 4]{L-444} that if we define $\pi_{x}\left(  a\right)
=\sup\limits_{n}\left(  a\wedge nx\right)  $ for $a$ in $X_{+}^{s}$ and
$\pi_{x}\left(  a\right)  =\pi_{x}\left(  a^{+}\right)  -\pi_{x}\left(
a^{-}\right)  $ for $a\in X,$ then $\pi_{a}$ is the band projection
$P_{\pi_{a}\left(  e\right)  }.$ We will simply write $P_{x}=P_{\pi_{x}\left(
e\right)  }$ and $B_{x}=R\left(  P_{x}\right)  $ the range of $P_{x}.$ Notice
that for every $x\in X^{s}$ we have $x+\infty=\infty.$ In particular; if $B$
is a band then for every $x\in B,$ $x+\infty_{B}=\infty_{B}.$ For $x\in X^{s}$
we can define its positive and negative parts as $x^{+}=x\vee0$ and
$x^{-}=-\left(  x\wedge0\right)  .$ Then $x^{+}-x^{-}=x.$ (The formula
$a\wedge b+a\vee b=a+b$ is still true in $X^{s}$). These parts can be
characterized by the following property: if $x=a-b$ with $a,b\in X_{+}^{s}$
and $a\wedge b=0,$ then $a=x^{+}$ and $b=x^{-}.$ Indeed we have%
\[
x^{+}=x\vee0=(a-b)\vee0=a\vee b-b=a+b-b=a.
\]
Now as $b\wedge a=b\wedge x^{+}=0$ the equality%
\[
P_{x^{+}}^{d}x=P_{x^{+}}^{d}\left(  x^{+}-x^{-}\right)  =P_{x^{+}}^{d}\left(
a-b\right)
\]
gives $b=x^{-}$ as well. Recall that tow elements in $X_{+}^{s}$ are said to
be disjoint and we write $x\perp y$ if $x\wedge y=0.$

\begin{lemma}
\label{YY2-A}Let $\left(  x_{\alpha}\right)  ,\left(  y_{\alpha}\right)  $ be
two nets in $X_{+}^{s}$ such that $\left(  x_{\alpha}\right)  \perp\left(
y_{\alpha}\right)  .$ Then the following statements hold.

\begin{enumerate}
\item[(i)] $%
{\textstyle\bigvee\limits_{\alpha}}
(x_{\alpha}+y_{\alpha})=%
{\textstyle\bigvee\limits_{\alpha}}
x_{\alpha}+%
{\textstyle\bigvee\limits_{\alpha}}
y_{\alpha}$ and $\bigwedge\limits_{\alpha}(x_{\alpha}+y_{\alpha}%
)=\bigwedge\limits_{\alpha}x_{\alpha}+\bigwedge\limits_{\alpha}y_{\alpha};$

\item[(ii)] $\limsup(x_{\alpha}+y_{\alpha})=\limsup x_{\alpha}+\limsup
y_{\alpha}$ and $\liminf(x_{\alpha}+y_{\alpha})=\liminf x_{\alpha}+\liminf
y_{\alpha}.$
\end{enumerate}
\end{lemma}

\begin{proof}
Put $x=%
{\textstyle\bigvee\limits_{\alpha}}
x_{\alpha}$ and $y=%
{\textstyle\bigvee\limits_{\alpha}}
y_{\alpha}.$ It follows from \cite[Lemma 11.(iii)]{L-900}, that $%
{\textstyle\bigvee\limits_{\alpha}}
x_{\alpha}\wedge%
{\textstyle\bigvee\limits_{\alpha}}
y_{\alpha}=0$ and hence $P_{x+y}=P_{x}+P_{y}.$

(i)The inequality $x+y\geq%
{\textstyle\bigvee\limits_{\alpha}}
\left(  x_{\alpha}+y_{\alpha}\right)  $ is obvious. On the other hand we have
$%
{\textstyle\bigvee\limits_{\alpha}}
\left(  x_{\alpha}+y_{\alpha}\right)  \geq x$ and $%
{\textstyle\bigvee\limits_{\alpha}}
\left(  x_{\alpha}+y_{\alpha}\right)  \geq y,$ which gives%
\[%
{\textstyle\bigvee\limits_{\alpha}}
\left(  x_{\alpha}+y_{\alpha}\right)  \geq x\vee y=x+y,
\]
where the last equality holds because $x\wedge y=0.$

For the second part we have clearly%
\[
\bigwedge\limits_{\alpha}(x_{\alpha}+y_{\alpha})\geq\bigwedge\limits_{\alpha
}x_{\alpha},\bigwedge\limits_{\alpha}y_{\alpha},
\]
and then as $\bigwedge\limits_{\alpha}x_{\alpha}$ and $\bigwedge
\limits_{\alpha}y_{\alpha}$ are disjoint we get%
\[
\bigwedge\limits_{\alpha}(x_{\alpha}+y_{\alpha})\geq\bigwedge\limits_{\alpha
}x_{\alpha}+\bigwedge\limits_{\alpha}y_{\alpha},
\]
On the other hand we have%
\[
P_{x}\bigwedge\limits_{\alpha}(x_{\alpha}+y_{\alpha})\leq x_{\beta,}%
\]
for all $\beta$ and then $P_{x}\bigwedge\limits_{\alpha}(x_{\alpha}+y_{\alpha
})\leq\bigwedge\limits_{\alpha}x_{\alpha}.$ Similarly we get $P_{y}%
\bigwedge\limits_{\alpha}(x_{\alpha}+y_{\alpha})\leq\bigwedge\limits_{\alpha
}y_{\alpha}$ and so%
\[
\bigwedge\limits_{\alpha}(x_{\alpha}+y_{\alpha})=P_{x}\bigwedge\limits_{\alpha
}(x_{\alpha}+y_{\alpha})+P_{y}\bigwedge\limits_{\alpha}(x_{\alpha}+y_{\alpha
})\leq\bigwedge\limits_{\alpha}x_{\alpha}+\bigwedge\limits_{\alpha}y_{\alpha
},
\]
\newline which ends the proof of (i).

(ii) This follows easily from (i).
\end{proof}

The above lemma is not valid if we have only $x_{\alpha}\perp y_{\alpha}$ for
each $\alpha.$ Take, for example, $X=%
\mathbb{R}
^{2}$, $x_{1}=\left(  1,0\right)  =y_{2}$ and $x_{2}=\left(  0,1\right)
=$\ $y_{1}.$

The following lemma gives another case when equalities in Lemma \ref{YY2-A}%
.(i) hold.

\begin{lemma}
\label{YY2-Q}Let $\left(  x_{\alpha}\right)  _{\alpha\in A}$ and $\left(
y_{\alpha}\right)  _{\alpha\in A}$ be two decreasing nets in $X_{+}^{s}$ then
$\inf(x_{\alpha}+y_{\alpha})=\inf x_{\alpha}+\inf y_{\alpha}.$
\end{lemma}

\begin{proof}
We will make use of \ref{YY2-k}.(ii) where the equality is proved if one of
the nets is constant. First observe that the inequality%
\[
\inf(x_{\alpha}+y_{\alpha})\geq\inf(x_{\alpha})+\inf(y_{\alpha})
\]
is quite obvious. Fix $\beta$ in $A.$ Then for any $\alpha\geq\beta$ we have%
\begin{align*}
\inf\limits_{\alpha\in A}(x_{\alpha}+y_{\alpha})  &  =\inf\limits_{\alpha
\geq\beta}(x_{\alpha}+y_{\alpha})\leq\inf\limits_{\alpha\geq\beta}\left(
x_{\alpha}+y_{\beta}\right) \\
&  =\inf\limits_{\alpha\geq\beta}x_{\alpha}+y_{\beta}=\inf\limits_{\alpha\in
A}x_{\alpha}+y_{\beta}.
\end{align*}
\newline Hence%
\[
\inf(x_{\alpha}+y_{\alpha})\leq\inf_{\beta}\left(  \inf\nolimits_{\alpha
}x_{\alpha}+y_{\beta}\right)  =\inf(x_{\alpha})+\inf(y_{\beta}).
\]
This completes the proof.
\end{proof}

\begin{lemma}
\label{X1}Let $\left(  x_{\alpha}\right)  _{\alpha\in A}$ be a net in
$X_{+}^{s}$ and $\left(  B_{\alpha}\right)  _{\alpha\in A}$ a net in
$\mathcal{B}\left(  X\right)  $ such that $x_{\alpha}\in B_{\alpha}^{s}$ for
every $\alpha\in A.$ Then $\sup x_{\alpha}\in\left(  \sup B_{\alpha}\right)
^{s},$ $\inf x_{\alpha}\in\left(  \inf B_{\alpha}\right)  ^{s},$ $\limsup
x_{\alpha}\in\left(  \limsup B_{\alpha}\right)  ^{s}$ and $\liminf x_{\alpha
}\in\left(  \liminf B_{\alpha}\right)  ^{s}.$
\end{lemma}

\begin{proof}
The statements are obvious if $x_{\alpha}\in X_{+}^{u}$ for every $\alpha\in
A.$ Now let $y$ be fixed in $X_{+}$ and observe that $y\wedge x_{\alpha}\in
B_{\alpha}$ for all $\alpha.$ So $y\wedge\sup x_{\alpha}=\sup\left(  y\wedge
x_{\alpha}\right)  \in\left(  \sup B_{\alpha}\right)  ^{s}.$ As this happens
for each $y\in X_{+}$ we get $\sup x_{\alpha}\in\left(  \sup B_{\alpha
}\right)  ^{s}.$ The proof of the other results is similar.
\end{proof}

\begin{remark}
\label{YY2-j}If $\left\{  x_{\alpha}:\alpha\in A\right\}  $ is a subset of
$X_{+}^{s}$ and $y\in X_{+}^{s}$ then $\sup\limits_{\alpha\in A}yx_{\alpha
}=y\sup\limits_{\alpha\in A}x_{\alpha}$ holds in $X_{+}^{s}.$ This follows
from \cite[Lemma 24]{L-900} when $A$ is finite and then holds for arbitrary
subsets using \cite[Lemma 23]{L-900}. It should be noted that a similar
formula for infimum fails in general (see Lemma \ref{YY2-E} below).
\end{remark}

\section{More about finite and infinite parts}

We develop in this section some material concerning the space $X^{s},$ the
sup-completion of $X,$ that are needed to prove our results in Section
\ref{MM}. These results can be interesting in their own right.

Let $X$ be a Dedekind complete Riesz space with weak order unit $e.$ It was
shown in \cite{L-900} that every element $y\in X_{+}^{s}$ has a decomposition:%

\[
y=y^{f}+y^{\infty}\in X^{s},
\]
where $y^{\infty}$ is the largest element in $B^{s}$ for some band $B$ in $X$
and $y^{f}\in B^{d}.$ It is easy to see that $x^{\infty}\leq y^{\infty}$
whenever $x\leq y$ in $X_{+}^{s},$ but it is not the case for the finite parts
in general. Consider for example $x=(1,1)\leq y=(1,\infty)$ in $\left(
\mathbb{R}
^{2}\right)  ^{s}.$

We would like to note this useful point for further reference.

\begin{remark}
\label{YY2-v}Elements of $X_{+}^{s}$ of the form $x^{\infty}$ are
characterized by the following property:%
\[
0<a\leq x^{\infty}\Longrightarrow na\leq x^{\infty}\text{ for all }%
n\in\mathbb{N}.
\]
Additionally, it is noteworthy to observe that if $P=P_{B}$ is a band
projection such that $Px=\infty_{B}$ and $P^{d}x\in X^{u}$ then $Px=x^{\infty
}$ and $P^{d}x=x^{f}.$
\end{remark}

\begin{lemma}
\label{YY2-H}Let $X$ be a Dedekind complete Riesz space and $x,y\in X_{+}%
^{s}.$ Then the following statements hold.

\begin{enumerate}
\item[(i)] $\left(  x+y\right)  ^{\infty}=x^{\infty}+y^{\infty}$ \ and
\ $\left(  x+y\right)  ^{f}\leq x^{f}+y^{f}$ with equality if $x$ and $y$ are disjoint.

\item[(ii)] $\left(  x\vee y\right)  ^{\infty}=x^{\infty}\vee y^{\infty
}=x^{\infty}+y^{\infty}$ and $\left(  x\vee y\right)  ^{f}=P_{y^{\infty}}%
^{d}x^{f}\vee P_{x^{\infty}}^{d}y^{f}\leq x^{f}\vee y^{f}.$

\item[(iii)] $\left(  x\wedge y\right)  ^{\infty}=x^{\infty}\wedge y^{\infty}$
and $\left(  x\wedge y\right)  ^{f}=x^{f}\wedge y^{f}+x^{f}\wedge y^{\infty
}+x^{\infty}\wedge y^{f}.$

\item[(iv)] $\left(  x.y\right)  ^{\infty}=x^{f}.y^{\infty}+x^{\infty}%
.y^{f}+x^{\infty}.y^{\infty}$ and $\left(  xy\right)  ^{f}=x^{f}y^{f}.$ In
particular, if $x\in X_{+}^{u}$ then $\left(  xy\right)  ^{f}=xy^{f}$ and
$\left(  xy\right)  ^{\infty}=xy^{\infty}.$
\end{enumerate}
\end{lemma}

\begin{proof}
Let $B$ be the band generated by $x^{\infty}+y^{\infty},$ so that $\infty
_{B}=x^{\infty}+y^{\infty}=x^{\infty}\vee y^{\infty}$, and let $P$ be the
corresponding band projection.

(i) Clearly, $P^{d}\left(  x+y\right)  =P^{d}\left(  x^{f}+y^{f}\right)  \in
X^{u},$ and then%
\[
P\left(  x+y\right)  \geq x^{\infty}+y^{\infty}=\infty_{B}.
\]
So $P\left(  x+y\right)  =x^{\infty}+y^{\infty}=\infty_{B}$ and then
$x^{\infty}+y^{\infty}=\infty_{B}=\left(  x+y\right)  ^{\infty}$ and%
\[
\left(  x+y\right)  ^{f}=P^{d}\left(  x^{f}+y^{f}\right)  =P_{y^{\infty}}%
^{d}x^{f}+P_{x^{\infty}}^{d}y^{f}\leq x^{f}+y^{f}.
\]

(ii) Again as $\infty_{B}\geq P\left(  x\vee y\right)  \geq x^{\infty}\vee
y^{\infty}=\infty_{B}$ we get
\[
\infty_{B}=P\left(  x\vee y\right)  =x^{\infty}\vee y^{\infty}.
\]
On the other hand%
\[
P^{d}\left(  x\vee y\right)  =P^{d}x\vee P^{d}y=P^{d}x^{f}\vee P^{d}%
y^{f}=P_{y^{\infty}}^{d}x^{f}\vee P_{x^{\infty}}^{d}y^{f}\in X^{u}.
\]
\newline This shows that%
\[
\infty_{B}=\left(  x+y\right)  ^{\infty}=\left(  x\vee y\right)  ^{\infty
},\text{ and }\left(  x\vee y\right)  ^{f}=P_{y^{\infty}}^{d}x^{f}\vee
P_{x^{\infty}}^{d}y^{f}.
\]
If $x\perp y$ then $x^{f}+y^{f}\perp x^{\infty}+y^{\infty}$ and then%
\[
P^{d}\left(  x^{f}+y^{f}\right)  =x^{f}+y^{f}.
\]

(iii) It follows from \cite[Lemma 12]{L-900} that%
\[
x\wedge y=x^{\infty}\wedge y^{\infty}+x^{f}\wedge y^{\infty}+x^{\infty}\wedge
y^{f}+x^{f}\wedge y^{f}.
\]
Considering $x^{\infty}\wedge y^{\infty}$ is infinite (unless zero) and
$x^{f}\wedge y^{\infty}+x^{\infty}\wedge y^{f}+x^{f}\wedge y^{f}$ is finite,
mutually disjoint, they are likely the infinite and finite parts of $x\wedge
y.$

(iv) The proof is similar.
\end{proof}

\begin{remark}
\label{YY2-o}As mentioned earlier the map $x\longmapsto x^{\infty}$ is
increasing on $X_{+}^{s},$ whereas the map $x\longmapsto x^{f}$ is not.
However, there is an important case where the implication: $x\leq
y\Longrightarrow x^{f}\leq y^{f}$ holds true. This occurs when the difference
is finite: If $y=x+a$ with $a\in X_{+}^{u}$ and $x\leq y,$ then $x^{f}\leq
y^{f}.$ Indeed we have%
\[
x=y^{\infty}+y^{f}-a=y^{\infty}+y^{f}-P_{y^{\infty}}a-P_{y^{\infty}}%
^{d}a=y^{\infty}+y^{f}-P_{y^{\infty}}^{d}a.
\]
But as $y^{f}-P_{y^{\infty}}^{d}a\in B_{y^{\infty}}^{d},$ we deduce from the
uniqueness of the decomposition \cite[Theorem 15]{L-900} that $x^{f}%
=y^{f}-P_{y^{\infty}}^{d}a\leq y^{f}.$
\end{remark}

\begin{remark}
\label{YY2-l}(i) It is well known that for every $x,y\in X_{+}$ we have
$B_{xy}=B_{x\wedge y}=B_{x}\cap B_{y}.$ This formula is still valid when
$x,y\in X_{+}^{s}.$ This can be shown by taking two nets $\left(  x_{\alpha
}\right)  $ and $\left(  y_{\alpha}\right)  $ in $X$ such that $x_{\alpha
}\uparrow x$ and $y_{\alpha}\uparrow y.$\newline(ii) It was shown in
\cite[Proposition 25]{L-900} that if $x\in X_{+}^{s}$ and $B$ is a projection
band then $\infty_{B}.x=\infty_{P_{B}x}=\infty_{B\cap B_{x}}.$ In particular,
if $B\subseteq B_{x}$ we have $\infty_{B}.x=\infty_{B}.$
\end{remark}

\begin{proposition}
\label{YY2-k}Let\textbf{\ }$\left(  x_{\alpha}\right)  _{\alpha\in A}$ be a
net in $X^{s}$ and let $y\in X^{s}.$ Then the following statements hold.

\begin{enumerate}
\item[(i)] $\sup(y+x_{\alpha})=y+\sup x_{\alpha}.$

\item[(ii)] If\textbf{\ }$(x_{\alpha})$ is order bounded from below in $X^{s}%
$, then $\inf\limits_{\alpha}(y+x_{\alpha})=y+\inf\limits_{\alpha}x_{\alpha}.$

\item[(iii)] $\limsup\limits_{\alpha}(y+x_{\alpha})=y+\limsup\limits_{\alpha
}x_{\alpha}.$

\item[(iv)] If\textbf{\ }$(x_{\alpha})$ is order bounded from below in $X^{s}%
$, then
\[
\liminf\limits_{\alpha}(y+x_{\alpha})=y+\liminf\limits_{\alpha}x_{\alpha}.
\]

\end{enumerate}
\end{proposition}

\begin{proof}
(i) This is a particular case of \cite[Property (P8)]{L-900}.

(ii) The inequality $y+\inf\limits_{\alpha}x_{\alpha}\leq\inf\limits_{\alpha
}\left(  y+x_{a}\right)  $ is obvious. For the converse assume first that
$y\in X^{u}.$ Then by the first inequality%
\[
-y+\inf\limits_{\alpha}\left(  y+x_{\alpha}\right)  \leq\inf\limits_{\alpha
}x_{\alpha},
\]
and so%
\[
\inf\limits_{\alpha}\left(  y+x_{\alpha}\right)  =y+\inf\limits_{\alpha
}x_{\alpha}.
\]
This shows the result for this particular case. Moreover, as $\left(
x_{\alpha}\right)  _{\alpha\in A}$ is order bounded from below we can assume
without loss of generality that $\left(  x_{\alpha}\right)  _{\alpha\in A}$
and $y$ are in the positive cone $X_{+}^{s}.$ We treat now the case
$y=\infty_{B}$ for some band $B.$ Let $P$ denotes the corresponding band
projection. Then from the inequality%
\[
\inf\limits_{\alpha}(x_{\alpha}+\infty_{B})\leq x_{\beta}+\infty_{B}%
,\qquad\beta\in A,
\]
\newline we deduce that%
\[
P^{d}\inf\limits_{\alpha}(x_{\alpha}+\infty_{B})\leq P^{d}x_{\beta}\leq
x_{\beta}.
\]
As this happens for every $\beta$ we get%
\[
P^{d}\inf\limits_{\alpha}(x_{\alpha}+\infty_{B})\leq\inf\limits_{\alpha
}x_{\alpha}.
\]
Now observe that%
\[
\inf\limits_{\alpha}(x_{\alpha}+\infty_{B})=P^{d}\inf\limits_{\alpha
}(x_{\alpha}+\infty_{B})+P\inf\limits_{\alpha}(x_{\alpha}+\infty_{B})\leq
\inf\limits_{\alpha}x_{\alpha}+\infty_{B},
\]
\newline which shows the second inequality. Finally the general case can be
derived by employing the decomposition $y=y^{f}+y^{\infty}$ in the following
way:%
\[
\inf\limits_{\alpha}(y+x_{\alpha})=y^{\infty}+\inf\limits_{\alpha}\left(
y^{f}+x_{\alpha}\right)  =y^{\infty}+y^{f}+\inf\limits_{\alpha}\left(
x_{\alpha}\right)  =y+\inf\limits_{\alpha}x_{\alpha}.
\]

(iii) and (iv) can be deduced easily from (i) and (ii).
\end{proof}

It follows from \cite[Theorem 5]{L-900} that if $\left(  x_{\alpha}\right)  $
and $\left(  y_{\alpha}\right)  $ are two nets in $X_{+}^{s}$ such that
$x_{\alpha}\uparrow x$ and $y_{\alpha}\uparrow y$ in $X^{s}$ then $\left(
x_{\alpha}+y_{\alpha}\right)  _{\alpha}\uparrow x+y$ (apply (i) to the map
$X\times X\longrightarrow X;$ $\left(  x,y\right)  \longmapsto x+y$).

\begin{proposition}
\label{YY2-T}Let $\left(  x_{\alpha}\right)  $ and $\left(  y_{\alpha}\right)
$ be two nets in $X^{s}$ that are bounded from below. Then the following
statements hold.

\begin{enumerate}
\item[(i)] $\liminf x_{\alpha}+\liminf y_{\alpha}\leq\liminf\left(  x_{\alpha
}+y_{\alpha}\right)  \leq\liminf x_{\alpha}+\limsup y_{\alpha}$

\item[(ii)] $\limsup\left(  x_{\alpha}+y_{\alpha}\right)  \leq\limsup
x_{\alpha}+\limsup y_{\alpha}.$

\item[(iii)] If $\lim y_{\alpha}$ exists then $\liminf\left(  x_{\alpha
}+y_{\alpha}\right)  =\liminf x_{\alpha}+\lim y_{\alpha}.$
\end{enumerate}
\end{proposition}

\begin{proof}
(i) We will make use of Lemma \ref{YY2-k}. We have for all $\beta\geq\theta,$%
\begin{align*}
\inf\limits_{\alpha\geq\beta}x_{\alpha}+\inf\limits_{\alpha\geq\beta}%
y_{\alpha}  &  \leq\inf\limits_{\alpha\geq\beta}(x_{\alpha}+y_{\alpha}%
)\leq\inf\limits_{\alpha\geq\beta}\left(  x_{\alpha}+\sup_{\alpha\geq\theta
}y_{\alpha}\right) \\
&  =\inf_{\alpha\geq\beta}x_{\alpha}+\sup_{\alpha\geq\theta}y_{\alpha}\leq
\lim\inf x_{\alpha}+\sup_{\alpha\geq\theta}y_{\alpha}.
\end{align*}
Taking the supremum over $\beta,$ we obtain%
\[
\liminf(x_{\alpha})+\liminf(y_{\alpha})\leq\liminf(x_{\alpha}+y_{\alpha}%
)\leq\liminf x_{\alpha}+\sup_{\alpha\geq\theta}y_{\alpha}.
\]
Then taking the infimum over $\theta$ and using Proposition \ref{YY2-k}\ we
get the desired inequalities.

(ii) We have for each $\beta,$%
\[
\sup\limits_{\alpha\geq\beta}\left(  x_{\alpha}+y_{\alpha}\right)  \leq
\sup\limits_{\alpha\geq\beta}x_{\alpha}+\sup\limits_{\alpha\geq\beta}%
y_{\alpha}.
\]
Then, taking the infimum over $\beta$ an using Lemma \ref{YY2-Q}, we get the
desired inequality.

(iii) This is an easy consequence of (i).
\end{proof}

\begin{proposition}
\label{YY2-E}Let $\left(  x_{\alpha}\right)  $ be a net in $X_{+}^{s},$ $u\in
X_{+}^{s}$ and $B\in\mathcal{B}\left(  X\right)  .$ Then the following
statements hold.

\begin{enumerate}
\item[(i)] If $u^{\infty}\in B_{\inf_{\alpha}x_{\alpha}}^{s}$\ then%
\[
\inf\limits_{\alpha}ux_{\alpha}=u\inf\limits_{\alpha}x_{\alpha}.
\]
In particular we have:

\begin{enumerate}
\item If $B\subset B_{\inf\limits_{a}x_{\alpha}}$ then\ $\inf\limits_{a}%
(\infty_{B}x_{\alpha})=\infty_{B}.\inf\limits_{a}x_{\alpha}=\infty_{B}.$

\item If $u\in X_{+}^{u}$ then $\inf\limits_{\alpha}ux_{\alpha}=u\inf
\limits_{\alpha}x_{\alpha}.$
\end{enumerate}

\item[(ii)] If $u^{\infty}\subset B_{\limsup x_{\alpha}}^{s}$ then
$\limsup\limits_{\alpha}(ux_{\alpha})=u\limsup\limits_{\alpha}\left(
x_{\alpha}\right)  .$ In particular, if $B\subset B_{\limsup\limits_{\alpha
}\left(  x_{\alpha}\right)  }$\ then $\limsup\limits_{\alpha}(x_{\alpha}%
\infty_{B})=\infty_{B}\limsup\limits_{\alpha}\left(  x_{\alpha}\right)  .$

\item[(iii)] If $u^{\infty}\in B_{\liminf\limits_{\alpha}x_{\alpha}}$
then\textbf{\ }\ $\liminf\limits_{\alpha}(ux_{\alpha})=u\liminf\limits_{\alpha
}x_{\alpha}.$ In particular, if $B\subset B_{\liminf\limits_{\alpha}x_{\alpha
}}$\ then%
\[
\liminf\limits_{\alpha}(\infty_{B}x_{\alpha})=\infty_{B}\liminf\limits_{\alpha
}\left(  x_{\alpha}\right)  =\infty_{B}.
\]

\end{enumerate}
\end{proposition}

\begin{proof}
(i)(a) Assume first that $B\subseteq B_{\inf\limits_{\alpha}x_{\alpha}}.$ Then
$\infty_{B}\inf\limits_{\alpha}x_{\alpha}=\infty_{B}=\infty_{B}x_{\beta}$ for
each $\beta\in A.$ Thus the formula%
\[
\inf\limits_{\alpha}(x_{\alpha}\infty_{B})=\infty_{B}\inf\limits_{\alpha
}x_{\alpha}=\infty_{B}%
\]
holds.

(b) Assume now that $u\in X_{+}^{u}.$ The inequality $\inf\limits_{\alpha
}\left(  ux_{\alpha}\right)  \geq u\inf\limits_{\alpha}x_{\alpha}$ is obvious.
For the converse let $z\in X_{+}^{u}$ such that $z\leq\inf_{\alpha}%
(ux_{\alpha})$. Then $z\in B_{ux_{\alpha}}\subset B_{u}=B_{u^{\ast}}.$ Hence
$u^{\ast}z\leq u^{\ast}ux_{\alpha}\leq x_{\alpha}$ for every $\alpha.$ It
follows that $u^{\ast}z\leq\inf x_{\alpha}$ and then $z=u.u^{\ast}z\leq
u\inf_{\alpha}x_{\alpha}$. From this we deduce the inequality $\inf_{\alpha
}(ux_{\alpha})\leq u\inf_{\alpha}x_{\alpha}.$

(c) The general case. Assume now that $u\in X_{+}^{s}.$ Since $(u^{f}%
x_{\alpha})_{\alpha}\perp(u^{\infty}x_{\alpha})_{\alpha}$\ it follows from
Lemma \ref{YY2-A} and cases (a) and (b) that%
\begin{align*}
\inf\limits_{\alpha}(ux_{\alpha})  &  =\inf\limits_{\alpha}(u^{f}x_{\alpha
})+\inf\limits_{\alpha}(u^{\infty}x_{\alpha})\\
&  =u^{f}\inf\limits_{\alpha}(x_{\alpha})+u^{\infty}\inf\limits_{\alpha
}(x_{\alpha})=u\inf\limits_{\alpha}(x_{\alpha}),
\end{align*}
\newline as required.

(ii) We have for each $\beta\in A,$%
\[
\sup_{\alpha\geq\beta}(x_{\alpha}\infty_{B})=\infty_{B}.\sup_{\alpha\geq\beta
}(x_{\alpha})=\infty_{B}.
\]
Since $B\subset B_{\limsup x_{\alpha}}$ it follows by (i) that%
\begin{align*}
\infty_{B}\limsup\limits_{\alpha}(x_{\alpha})  &  =\infty_{B}\inf
\limits_{\beta}\left(  \sup\limits_{\alpha\geq\beta}(x_{\alpha})\right)
=\inf\limits_{\beta}\left(  \infty_{B}.\sup\limits_{\alpha\geq\beta}%
(x_{\alpha})\right) \\
&  =\limsup\limits_{\alpha}(x_{\alpha}.\infty_{B})=\infty_{B}.
\end{align*}
The general case can be deduced in a similar way as in (i).

(iii) Assume first that $u=\infty_{B}$ for some $B\subseteq B_{\liminf
\limits_{\alpha}x_{\alpha}}^{s}.$ One inequality is obvious as $\liminf\left(
\infty_{B}x_{\alpha}\right)  \leq\infty_{B}=\infty_{B}\liminf x_{\alpha}.$ To
prove the converse let us put $u_{\beta}=u\inf\limits_{\alpha\geq\beta
}x_{\alpha}$ for $\beta\in A.$ Then $u_{\beta}\uparrow\infty_{B}.$ So for
every $\gamma\geq\beta,$ $u_{\beta}\in B_{\inf\limits_{\alpha\geq\gamma
}x_{\alpha}}.$ It follows in view of (i) that%
\[
\inf\limits_{\alpha\geq\gamma}u_{\beta}x_{\alpha}=u_{\beta}\inf\limits_{\alpha
\geq\gamma}x_{\alpha}=u_{\beta}.
\]
By taking the supremum over $\gamma$ we get%
\begin{align*}
u_{\beta}  &  =\liminf\limits_{\alpha}u_{\beta}x_{\alpha}=\sup\limits_{\gamma
}\inf\limits_{\alpha\geq\gamma}u_{\beta}x_{\alpha}=\sup\limits_{\beta}%
u_{\beta}\inf\limits_{\alpha\geq\beta}x_{\alpha}\\
&  =u_{\beta}\sup\limits_{\beta}\inf\limits_{\alpha\geq\beta}x_{\alpha
}=u_{\beta}\liminf\limits_{\alpha}x_{\alpha}.
\end{align*}
\newline Taking the supremum over $\beta$ we get%
\[
\infty_{B}=\infty_{B}\liminf\limits_{\alpha}x_{\alpha}=\sup\limits_{\beta
}\liminf\limits_{\alpha}u_{\beta}x_{\alpha}\leq\liminf\limits_{\alpha}%
\infty_{B}x_{\alpha}.
\]
This proves (iii) in that special case. The general case can be deduced as in (i).
\end{proof}

\begin{remark}
\label{YY2-W}In Proposition \ref{YY2-E}, the condition $B\subseteq
B_{\inf_{\alpha}x_{\alpha}}$ can not be dropped as the following example can
show. If $X=%
\mathbb{R}
,$ $u=\infty$ and $x_{n}=n^{-1},$ $n\geq1,$ then $\infty=\inf ax_{n}\neq u\inf
x_{n}=0.$ But it is useful to note the following inequality $\inf
\limits_{\alpha}\left(  ux_{\alpha}\right)  \leq u^{\infty}+u^{f}\inf
x_{\alpha}.$
\end{remark}

\begin{lemma}
\label{X4}Let $\left(  x_{\alpha}\right)  _{\alpha\in A},$ $\left(  y_{\alpha
}\right)  _{\alpha\in A}$ be two nets in $X_{+}^{s}.$ Then%
\[
\limsup\limits_{\alpha\in A}\left(  x_{\alpha}y_{\alpha}\right)  \geq
\limsup\limits_{\alpha\in A}x_{\alpha}\liminf\limits_{\alpha\in A}y_{\alpha}.
\]
If, in addition, $\left(  \sup\nolimits_{\alpha\geq\beta}x_{\alpha}\right)
^{\infty}\in B_{\inf\limits_{\alpha\geq\beta}y_{\alpha}}^{s}$ for some $\beta$
and $\left(  \liminf\limits_{\alpha\in A}y_{\alpha}\right)  ^{\infty}\in
B_{\limsup x_{\alpha}}^{s}$ then%
\[
\liminf\limits_{\alpha}\left(  x_{\alpha}y_{\alpha}\right)  \leq
\limsup\limits_{\alpha}x_{\alpha}.\liminf\limits_{\alpha}y_{\alpha}.
\]

\end{lemma}

\begin{proof}
Fix $\beta,\gamma$ in $A$ with $\beta\geq\gamma.$ Then we have for each
$\theta\geq\beta,$%
\[
\sup\limits_{\alpha\geq\beta}\left(  x_{\alpha}y_{\alpha}\right)  \geq
x_{\theta}y_{\theta}\geq x_{\theta}\inf\limits_{\alpha\geq\beta}y_{\alpha}.
\]
According to Remark \ref{YY2-j} we have%
\begin{align*}
\sup\limits_{\alpha\geq\gamma}\left(  x_{\alpha}y_{\alpha}\right)   &
\geq\sup\limits_{\alpha\geq\beta}\left(  x_{\alpha}y_{\alpha}\right)  \geq
\sup\limits_{\theta\geq\beta}\left(  x_{\theta}\inf\limits_{\alpha\geq\beta
}y_{\alpha}\right) \\
&  =\sup\limits_{\theta\geq\beta}x_{\theta}.\inf\limits_{\alpha\geq\beta
}y_{\alpha}\geq\limsup x_{\alpha}.\inf\limits_{\alpha\geq\beta}y_{\alpha}.
\end{align*}
\newline Taking the supremum over $\beta$ we get%
\[
\sup\limits_{\alpha\geq\gamma}\left(  x_{\alpha}y_{\alpha}\right)  \geq
\sup\limits_{\beta\geq\gamma}\left(  \limsup x_{\alpha}.\inf\limits_{\alpha
\geq\beta}y_{\alpha}\right)  =\limsup x_{\alpha}.\liminf y_{\alpha}.
\]
From this we derive the inequality%
\[
\limsup\left(  x_{\alpha}y_{\alpha}\right)  \geq\limsup x_{\alpha}.\liminf
y_{\alpha}.
\]

(ii) Assume now $\left(  \sup\nolimits_{\alpha\geq\beta}x_{\alpha}\right)
^{\infty}\in B_{\inf\limits_{\alpha\geq\beta}y_{\alpha}}^{s}$ for some
$\beta\in A.$ Then
\[
\left(  \sup\nolimits_{\alpha\geq\gamma}x_{\alpha}\right)  ^{\infty}\in
B_{\inf\limits_{\alpha\geq\gamma}y_{\alpha}}^{s}\text{ for every }\gamma
\geq\beta.
\]
Now for $\theta\geq\gamma\geq\beta$ we have%
\[
\inf\limits_{\alpha\geq\gamma}\left(  x_{\alpha}y_{\alpha}\right)  \leq
y_{\theta}\sup\limits_{\alpha\geq\gamma}x_{\alpha}.
\]
It follows that%
\[
\inf\limits_{\alpha\geq\gamma}\left(  x_{\alpha}y_{\alpha}\right)  \leq
\inf\limits_{\theta\geq\gamma}\left(  y_{\theta}.\sup\limits_{\alpha\geq
\gamma}x_{\alpha}\right)  =\inf\limits_{\theta\geq\gamma}y_{\theta}%
.\sup\limits_{\alpha\geq\gamma}x_{\alpha}\leq\liminf y_{\alpha}.\sup
\limits_{\alpha\geq\gamma}x_{\alpha}.
\]
\newline where we have used \ref{YY2-E}.(i) in the equality above. For a fixed
$\gamma$ we have for every $\delta\geq\gamma,$%
\[
\inf\limits_{\alpha\geq\gamma}\left(  x_{\alpha}y_{\alpha}\right)  \leq
\inf\limits_{\alpha\geq\delta}\left(  x_{\alpha}y_{\alpha}\right)  \leq\liminf
y_{\alpha}.\sup\limits_{\alpha\geq\delta}x_{\alpha}.
\]
Now taking the infimum over $\delta\geq\gamma$ and using again Lemma
\ref{YY2-E}(i) we get%
\[
\inf\limits_{\alpha\geq\gamma}\left(  x_{\alpha}y_{\alpha}\right)  \leq
\inf\limits_{\beta\geq\gamma}\left(  \liminf y_{\alpha}.\sup\limits_{\alpha
\geq\beta}x_{\alpha}\right)  =\liminf y_{\alpha}.\limsup x_{\alpha}.
\]
as required.
\end{proof}

We conclude this section with a brief discussion on Boolean algebras. Recall
that a Boolean algebra is a distributive lattice $\mathcal{A}$ with smallest
and largest elements that is complemented. The latter means that for every
element $a\in\mathcal{A}$ there exists a (necessarily unique) element
$a^{\prime}$ such that $a\wedge a^{\prime}=0$ and $a\vee a^{\prime}=1,$ where
$0$ denotes the smallest element of $\mathcal{A}$ and $1$ its largest one. The
Boolean algebra $\mathcal{A}$ is said to be Dedekind complete if every
nonempty subset has a supremum.

Consider a Dedekind complete Riesz space $X$ with weak order unit $e.$ Three
crucial Boolean algebras in this context are isomorphic. The two first are
familiar: the set $\mathcal{C}\left(  e\right)  $ consisting of all components
of $e,$ and the set of all band projections $\mathcal{B}\left(  X\right)  .$
These are isomorphic through the mapping:

\begin{center}
$\mathcal{C}\left(  e\right)  \longrightarrow\mathcal{B}\left(  X\right)
;\qquad u\longmapsto B_{u}.$
\end{center}

It should be noted that this map preserves suprema and infima. Specifically,
for any set $\left\{  p_{\alpha}:\alpha\in A\right\}  $ of components of $e,$
$\sup B_{p_{\alpha}}=B_{\sup p_{\alpha}}$ and $\inf B_{p_{\alpha}}=B_{\inf
p_{\alpha}}.$ Observe that the first formula remains valid for general sets,
the second, however, fails in general. The third noteworthy Boolean algebra of
interest is similarly isomorphic to the aforementio\`{a}ned ones. It is
intricately associated to the space $X^{s}$ as it is consisting of infinite
parts of positive elements within $X^{s}.$ Let us employ the following
notation to represent it:

\begin{center}%
\[
\infty\left(  X\right)  =\left\{  x^{\infty}:x\in X_{+}^{s}\right\}  =\left\{
\infty_{B}:B\in\mathcal{B}\left(  X\right)  \right\}  .
\]

\end{center}

The following result tells us that $\infty\left(  X\right)  $ is isomorphic to
$\mathcal{B}\left(  X\right)  .$

\begin{proposition}
\label{YY2-B}Let $\left(  B_{\alpha}\right)  _{\alpha\in A}$ be a net in
$\mathcal{B}\left(  X\right)  .$ The following hold.

\begin{enumerate}
\item[(i)] $\inf\limits_{\alpha}\infty_{B_{\alpha}}=\infty_{\inf
\limits_{\alpha}B_{\alpha}}$ and $\sup\limits_{\alpha}\infty_{B_{\alpha}%
}=\infty_{\sup\limits_{\alpha}B_{\alpha}}.$

\item[(ii)] $\liminf\limits_{\alpha}\left(  \infty_{B_{\alpha}}\right)
=\infty_{\liminf_{\alpha}B_{\alpha}}\ $and\textbf{ }$\limsup\limits_{\alpha
}\left(  \infty_{B_{\alpha}}\right)  =\infty_{\limsup_{\alpha}B_{\alpha}}.$

\item[(iii)] The map $\phi:\mathcal{B}(X)\longrightarrow\left\{  \infty
_{B}:B\in\mathcal{B}\left(  X\right)  \right\}  ;$ $B\longmapsto\infty_{B}$ is
an order continuous Boolean algebra isomorphism.
\end{enumerate}
\end{proposition}

\begin{proof}
(i) The inequality\ $%
{\textstyle\bigwedge_{\alpha}}
\infty_{B_{\alpha}}\geq\infty_{\wedge_{\alpha}B_{\alpha}}\ $is evident.
Conversely, if $z\in\left[
{\textstyle\bigwedge_{\alpha}}
\infty_{B_{\alpha}}\right]  ^{\leq},$ then $z\leq\infty_{B_{\alpha}}$ for
every $\alpha,$ so $z\in B_{\alpha}$ for every $\alpha,$ and consequently
$z\in%
{\textstyle\bigwedge_{\alpha}}
B_{\alpha}.$ Therefore $z\leq\infty_{\wedge_{\alpha}B_{\alpha}},$ establishing
the desired inequality. For the second result, it is clear that if $F$ is
finite then%
\[
\sup\limits_{\alpha\in F}\infty_{B_{\alpha}}=\infty_{\sup\limits_{\alpha\in
F}B_{\alpha}}.
\]
Now it is sufficient to observe that%
\[
\sup\limits_{\alpha}\infty_{B_{\alpha}}=\sup\limits_{F\text{ finite }\subseteq
A}\sup\limits_{\alpha\in F}\infty_{B_{\alpha}},
\]
and%
\[
\infty_{\sup\limits_{\alpha}B_{\alpha}}=\sup\limits_{F\text{ finite }\subseteq
A}\infty_{\sup\limits_{\alpha\in F}B_{\alpha}}.
\]

(ii) is an immediate consequence of (i).

(iii) The fact that $\phi$ is an isomorphism is quite clear, and since it
respects infinite suprema and infima, it is order continuous by (i) and (ii).
\end{proof}

\section{The star map}

According to \cite[Theorem 5]{L-931} any universally complete Riesz space $X$
with weak unit $e$ is von Neumann regular (that is, for every $a\in X$ there
exists $b\in X$ such that $a=a^{2}b$ ) and it is not difficult to deduce from
this result that any order weak unit element is invertible (see for example
\cite[Remark 3.3]{L-424}). Here, we present a concise proof of this result
employing the concept of sup-completion.

\begin{lemma}
\label{YY2-s}Let $X$ be a universally complete Riesz space with weak unit $e,$
which is also an algebraic unit. Then every weak unit $x$ has an inverse in
$X$.
\end{lemma}

\begin{proof}
Assume first that $x\in X_{+}.$ We know by \cite[Theorem 146.3]{b-1087} that
$x+\dfrac{1}{n}e$ is invertible. Let $y_{n}$ denotes its inverse. Then $y_{n}$
is increasing and if we put $y=\sup y_{n}\in X_{+}^{s}$ we have%
\begin{equation}
\left(  x+\dfrac{1}{n}e\right)  y_{n}=xy_{n}+\dfrac{1}{n}y=e. \tag{*}%
\end{equation}
In particular, $xy_{n}\leq e.$ By taking the supremum over $n$ we get $xy\leq
e.$ As $x$ is a weak unit we get $y^{\infty}=0,$ that is, $y\in X.$ Now taking
the limit as $n\longrightarrow\infty$ in $\left(  \ast\right)  $ we obtain
$xy=e$ and we are done.

For the general case\textbf{ }we write\textbf{ }$x=x^{+}-x^{-}.$ So $\left(
x^{+}-x^{-}\right)  y=e.$ Write $a=P_{x^{+}}y$ and $b=P_{x^{-}}y.$ Then%
\[
e=\left(  x^{+}-x^{-}\right)  \left(  a+b\right)  =ax^{+}-bx^{-}.
\]
So $x\left(  a-b\right)  =e$ and we are done.
\end{proof}

For every $y\in X_{+}$ the band $B_{y^{f}}$ in $X^{u}$ generated by the finite
part of $y$ is a universally complete Riesz space with unit $p=P_{y^{f}}e.$
The inverse of $y^{f}$ in the band $B_{y^{f}}$ will be denoted by $y^{\ast}.$
Thus we have the following:%

\[
yy^{\ast}=y^{\ast}y=e_{y^{f}}=e_{B_{y^{f}}}:=P_{y^{f}}e.
\]
In particular $B_{y^{\ast}}=B_{y^{f}}\subseteq B_{y}.$

We list now some useful properties of the map $x\longmapsto x^{\ast}$ defined
from $X^{s}$ to $X^{u}.$

\begin{proposition}
\label{YY2-Jm}Let $X$ be a Dedekind complete Riesz space and $x,y\in X^{s}$.
Then the following hold.

\begin{enumerate}
\item[(i)] $0^{\ast}=\infty_{B}^{\ast}=0$ for every band in $X.$

\item[(ii)] $\left(  \lambda x\right)  ^{\ast}=\lambda^{-1}x^{\ast}$ for every
real $\lambda\neq0.$

\item[(iii)] If $x\perp y,$ then $x^{\ast}\perp y^{\ast}$ and\textbf{\ }%
$(x+y)^{\ast}=x^{\ast}+y^{\ast}.\ $In particular, $\left\vert x\right\vert
^{\ast}=\left(  x^{+}\right)  ^{\ast}+\left(  x^{-}\right)  ^{\ast}$ and
$\left(  Qx\right)  ^{\ast}=Qx^{\ast}$ for every band projection $Q$.

\item[(iv)] $\left(  xy\right)  ^{\ast}=x^{\ast}y^{\ast}\in B_{x^{\ast}%
y^{\ast}}=B_{x^{\ast}\wedge y^{\ast}}.$

\item[(v)] $\left(  x^{p}\right)  ^{\ast}=\left(  x^{\ast}\right)  ^{p}$ for
$x\in X_{+}^{s}$ and $p>0.$

\item[(vi)] If $0\leq x\leq y$ and $P$ is a band projection such that $P\leq
P_{x}.$ Then $Py^{\ast}\leq Px^{\ast}.$ In particular $P_{x}y^{\ast}\leq
x^{\ast}.$
\end{enumerate}
\end{proposition}

\begin{proof}
We will show only the last property. Multiply the inequality $x\leq y$ by
$y^{\ast}x^{\ast},$ yields%
\[
P_{x^{\ast}}y^{\ast}\leq P_{y^{\ast}}x^{\ast}\leq x^{\ast}.
\]
\newline Furthermore, as $x\leq y,$ we have $x^{\infty}\perp y^{f},$ implying
$P_{x^{\infty}}y^{\ast}=0.$ Now, given that $P\leq P_{x}=P_{x^{\ast}%
}+P_{x^{\infty}},$ we deduce that%
\[
Py^{\ast}\leq P_{x^{\ast}}y^{\ast}+P_{x^{\infty}}y^{\ast}=P_{x^{\ast}}y^{\ast
},
\]
and then%
\[
Py^{\ast}\leq PP_{x^{\ast}}y^{\ast}\leq PP_{y^{\ast}}x^{\ast}\leq Px^{\ast}.
\]
\newline Thus, the desired inequality is established.
\end{proof}

\begin{lemma}
\label{YY2-q}Let $(x_{\alpha})_{\alpha\in A}$ be a net in $X_{+}^{u}$ and
$x\in X^{s}$ such that $x_{\alpha}\uparrow x,$ then $x_{\alpha}^{\ast}%
\overset{o}{\longrightarrow}x^{\ast}$ in $X^{u}$.
\end{lemma}

\begin{proof}
It is enough to prove the following inequalities:%
\[
\limsup x_{\alpha}^{\ast}\leq x^{\ast}\leq\liminf x_{\alpha}^{\ast}.
\]
According to Lemma\textbf{ }\ref{YY2-Jm} we infer from the inequalities
$x_{\alpha}\leq x$ (for $\alpha\in A$) that%
\[
P_{x_{\alpha}}x^{\ast}\leq x_{\alpha}^{\ast}.
\]
Since the net $\left(  x_{\alpha}\right)  $ is increasing, this implies that%
\begin{equation}
P_{x}x^{\ast}=x^{\ast}\leq\liminf x_{\alpha}^{\ast}. \label{H1}%
\end{equation}
On the other hand we have by Lemma \ref{YY2-Jm}.(iv) $P_{x_{\beta}}x_{\alpha
}^{\ast}\leq x_{\beta}^{\ast}$ for all $\alpha\geq\beta.$ So%
\[
\limsup\limits_{\alpha}P_{x_{\beta}}x_{\alpha}^{\ast}=P_{x_{\beta}}%
\limsup\limits_{\alpha}x_{\alpha}^{\ast}\leq x_{\beta}^{\ast}.
\]
Multiplying the above inequality by $x_{\beta},$ we obtain%
\[
x_{\beta}\limsup x_{\alpha}^{\ast}\leq x_{\beta}x_{\beta}^{\ast}\leq e.
\]
Taking the supremum over $\beta$ yields%
\begin{equation}
x\limsup x_{\alpha}^{\ast}\leq e. \label{H2}%
\end{equation}
In particular, $x^{\infty}\limsup x_{\alpha}^{\ast}\leq e,$ implying
$x^{\infty}\perp\limsup x_{\alpha}^{\ast}.$ As $\limsup x_{\alpha}^{\ast}$
belongs to $B_{x}^{s},$ we have%
\[
\limsup x_{\alpha}^{\ast}=P_{x^{\infty}}\limsup x_{\alpha}^{\ast}+P_{x^{f}%
}\limsup x_{\alpha}^{\ast}=P_{x^{\ast}}\limsup x_{\alpha}^{\ast}.
\]
Thus (\ref{H2}) implies that%
\begin{equation}
\limsup x_{\alpha}^{\ast}=P_{x^{\ast}}\limsup x_{\alpha}^{\ast}=x^{\ast
}x\limsup x_{\alpha}^{\ast}\leq x^{\ast}. \label{H3}%
\end{equation}
Now, combining (\ref{H1}) and (\ref{H3}), we conclude that $x_{\alpha}^{\ast
}\overset{o}{\longrightarrow}x^{\ast}$ in $X^{u}.$
\end{proof}

\begin{proposition}
\label{YY2-r}Let $\left(  x_{\alpha}\right)  _{\alpha\in A}$ and $\left(
y_{\alpha}\right)  _{\alpha\in A}$ be two sequences in $X_{+}^{u}$ such that
$x_{\alpha}\uparrow x$ and\ $y_{\alpha}\uparrow y$ in $X^{s}$ and $x_{\alpha
}^{2}\leq y_{\alpha}$ for all $\alpha.$ Then $x_{\alpha}y_{\alpha}^{\ast
}\overset{o}{\longrightarrow}x^{f}y^{\ast}\ $in $X^{u}.$
\end{proposition}

\begin{proof}
According to Lemma \ref{YY2-q}, we have that%
\begin{equation}
y_{\alpha}^{\ast}\overset{o}{\longrightarrow}y^{\ast}\text{ in }X^{u}.
\label{YY2-r-1}%
\end{equation}
Moreover, from the inequality $x_{\alpha}^{2}\leq y_{\alpha}$ it follows that%
\[
x_{\alpha}y_{\alpha}^{\ast}=e_{x_{\alpha}}x_{\alpha}y_{\alpha}^{\ast}\leq
x_{\alpha}^{\ast}.
\]
Notably, the sequence $\left(  x_{\alpha}y_{\alpha}^{\ast}\right)  $ is order
bounded in $X^{u}.$ Now as $x_{\alpha}\uparrow x$ we have $e_{x^{\ast}%
}x_{\alpha}\uparrow x^{f}\in X^{u}$ and thus%
\begin{equation}
P_{x^{\ast}}x_{\alpha}\overset{o}{\longrightarrow}x^{f}\text{ in }X^{u}.
\label{YY2-r-2}%
\end{equation}
As the product is order continuous in $X^{u}$ we derive from \ref{YY2-r-1} and
\ref{YY2-r-2} that $P_{x^{\ast}}\left(  x_{\alpha}y_{\alpha}^{\ast}\right)
\overset{o}{\longrightarrow}x^{f}y^{\ast}$ in $X^{u}.$ In particular%
\begin{align}
x^{f}y^{\ast}  &  =\limsup\left(  P_{x^{\ast}}\left(  x_{\alpha}y_{\alpha
}^{\ast}\right)  \right)  =P_{x^{\ast}}\limsup\left(  x_{\alpha}y_{\alpha
}^{\ast}\right) \label{YY2-r-3}\\
&  =\liminf\left(  P_{x^{\ast}}\left(  x_{\alpha}y_{\alpha}^{\ast}\right)
\right)  =P_{x^{\ast}}\liminf\left(  x_{\alpha}y_{\alpha}^{\ast}\right)
.\nonumber
\end{align}
To conclude our proof observe that both $\limsup\left(  x_{\alpha}y_{\alpha
}^{\ast}\right)  $ and $\liminf\left(  x_{\alpha}y_{\alpha}^{\ast}\right)  $
belong to the band $B_{x^{\ast}}$ in $X^{u},$ enabling us to simplify
(\ref{YY2-r-3}) by removing $P_{x^{\ast}}$ from the right side members.
\end{proof}

\begin{lemma}
\label{L1}Let $\left(  x_{\alpha}\right)  _{\alpha\in A}$ be an order bounded
net in $X_{+}^{u}$ and $u\in X_{+}^{u}.$ Then the following statements are valid.

\begin{enumerate}
\item[(i)] $\inf\limits_{\alpha}\left(  ux_{\alpha}\right)  =u\inf x_{\alpha}$
and $\sup\limits_{\alpha}\left(  ux_{\alpha}\right)  =u\sup x_{\alpha}$.

\item[(ii)] $\liminf\limits_{\alpha}\left(  ux_{\alpha}\right)  =u\liminf
\limits_{\alpha}x_{\alpha}$ and $\limsup\limits_{\alpha}\left(  ux_{\alpha
}\right)  =u\limsup\limits_{\alpha}x_{\alpha}.$
\end{enumerate}
\end{lemma}

\begin{proof}
(i) The inequality $\inf\limits_{\alpha}\left(  ux_{\alpha}\right)  \geq u\inf
x_{\alpha}$ is obvious. To establish the reverse inequality, let $z\in X$ be a
positive lower bound of $\left\{  ux_{\alpha}:\alpha\in A\right\}  .$ It is
enough to show that $z\leq u\inf x_{\alpha}.$ Observe that $u^{\ast}z\leq
u^{\ast}ux_{\alpha}\leq x_{\alpha}$ and then $u^{\ast}z\leq\inf x_{\alpha}$.
But as $z$ belongs to the band $B_{u}$ we have $z=uu^{\ast}z\leq u\inf
x_{\alpha}.$ This proves the first assertion; the proof of the second
assertion follows a similar line of reasoning.

(ii) This is an immediate consequence of (i).
\end{proof}

\begin{lemma}
\label{YY2-t}Let $\left(  x_{\alpha}\right)  $ and $\left(  y_{\alpha}\right)
$ be two order bounded nets in $X_{+}^{u}.$ Then the following inequalities
hold.%
\begin{align*}
\liminf x_{\alpha}.\liminf y_{\alpha}  &  \leq\liminf\left(  x_{\alpha
}y_{\alpha}\right)  \leq\liminf x_{\alpha}.\limsup y_{\alpha}\\
&  \leq\limsup\left(  x_{\alpha}y_{\alpha}\right)  \leq\limsup x_{\alpha
}.\limsup y_{\alpha}.
\end{align*}

\end{lemma}

\begin{proof}
(i) For any $\gamma$ large enough and $\beta\geq\gamma$ we have%
\[
\sup_{\alpha\geq\beta}\left(  x_{\alpha}y_{\alpha}\right)  \leq\sup
_{\alpha\geq\beta}x_{\alpha}.\sup_{\alpha\geq\gamma}y_{\alpha}.
\]
So taking the infimum over $\beta$ it follows from Lemma \ref{L1} that%
\[
\limsup\left(  x_{\alpha}y_{\alpha}\right)  \leq\limsup x_{\alpha}%
.\sup_{\alpha\geq\gamma}y_{\alpha}.
\]
Hence, taking the infimum over $\gamma$, we obtain%
\[
\limsup\left(  x_{\alpha}y_{\alpha}\right)  \leq\limsup\left(  x_{\alpha
}\right)  \limsup\left(  y_{\alpha}\right)  .
\]

(ii) For $\beta\geq\gamma$ we have%
\[
\inf_{\alpha\geq\beta}\left(  x_{\alpha}y_{\alpha}\right)  \geq\inf
_{\alpha\geq\beta}x_{\alpha}.\inf_{\alpha\geq\gamma}y_{\alpha}.
\]
So taking the supremum over $\beta$ it follows from Lemma \ref{L1} that%
\[
\liminf\left(  x_{\alpha}y_{\alpha}\right)  \geq\liminf x_{\alpha}%
.\inf_{\alpha\geq\gamma}y_{\alpha}.
\]
Hence, taking the supremum over $\gamma$, we obtain%
\[
\liminf\left(  x_{\alpha}y_{\alpha}\right)  \geq\liminf x_{\alpha}.\liminf
y_{\alpha}.
\]
The two other inequalities are similar.
\end{proof}

\begin{lemma}
\label{YY3-a}Let $\left(  x_{\alpha}\right)  $ and $\left(  y_{\alpha}\right)
$ be two order bounded sequences in $X_{+}^{u}.$ If $x_{\alpha}\overset
{o}{\longrightarrow}x$ then $\limsup\left(  x_{\alpha}y_{\alpha}\right)
=x\limsup y_{\alpha}$ and $\liminf\left(  x_{\alpha}y_{\alpha}\right)
=x\liminf y_{\alpha}.$
\end{lemma}

\begin{proof}
By Lemma \ref{YY2-t} we have%
\begin{align*}
x\liminf y_{\alpha}  &  =\liminf x_{\alpha}.\liminf y_{\alpha}\leq\liminf
x_{\alpha}y_{\alpha}\\
&  \leq\limsup x_{\alpha}.\liminf y_{\alpha}=x\liminf y_{\alpha}%
\end{align*}
Hence $\liminf\left(  x_{\alpha}y_{\alpha}\right)  =x\liminf y_{\alpha}.$ The
second equality follows similarly.
\end{proof}

We will prove now a multiplicative decomposition property in $X_{+}^{s}.$

\begin{proposition}
\label{YY2-P}Let $X$ be a Dedekind complete Riesz space and $x,y,z\in
X_{+}^{s}.$ If $x\leq yz$ then there exists a decomposition $x=ab$ of $x$ with
$0\leq a\leq y$ and $0\leq b\leq z.$
\end{proposition}

\begin{proof}
By restricting ourselves to the band $B_{y+z}$ we may assume that $X$ has a
weak unit $e>0.$ Assume first that $y$ is finite. Then $x=y.y^{\ast}x$ is a
suitable decomposition. Indeed as $x\leq yz$ we have $x\in B_{y},$ and so
$P_{y}x=x.$ Moreover $y^{\ast}x\leq y^{\ast}yz\leq z.$ If $y$ is infinite,
then a suitable decomposition could be $x=x\left(  z^{\ast}+e_{z^{\infty}%
}\right)  .\left(  z^{f}+e_{z^{\infty}}\right)  .$

General case. Write $x=Px+P^{d}x$ where $P=P_{y^{f}},$ and observe that
$Px\leq Pyz=y^{f}Pz$ and $P^{d}x\leq y^{\infty}P^{d}z.$ Using the previous two
cases we can write $Px=ab$ and $P^{d}x=a^{d}b^{d},$ with $a,b\in B_{y^{f}}%
^{s},$ $a^{d},b^{d}\in\left(  B_{y^{f}}^{d}\right)  ^{s}$ with $0\leq a\leq
Py,$ $0\leq b\leq Pz,$ $0\leq a^{d}\leq P^{d}y,$ and $0\leq b^{d}\leq P^{d}z.$
It is easily seen now that $x=\left(  a+a^{d}\right)  .\left(  b+b^{d}\right)
$ is a suitable decomposition of $x$.
\end{proof}

\section{Applications\label{MM}}

In this section, we consider a Riesz conditional triple $\left(  X,e,T\right)
$ , unless explicitely stated otherwise. Here $X$ is a Dedekind complete Riesz
space with order weak unit $e,$ and $T$ a conditional expectation operator on
$X$ with $Te=e.$

\begin{definition}
Let $X$ be an Archimedean Riesz space with weak order unit $e.$ Let $M=\left(
m_{ij}\right)  _{1\leq i,j\leq n}$ be an $n$-matrix with entries in $X.$ We
say that $M$ is positive semi-definite if $M$ is symmetric and $x^{T}Mx\in
X_{+}$ for every $x\in\left(  X_{e}\right)  ^{n},$ where $X_{e}$ is the ideal
generated by $e$ and $x^{T}=\left(  x_{1},....,x_{n}\right)  $ is the
transpose of $x.$ This means that $\sum\limits_{i,j}m_{i,j}x_{i}x_{j}\geq0$
for all $x_{1},...,x_{n}\in X_{e}.$
\end{definition}

\begin{remark}
It is clear that if $M$ is a positive semi-definite matrix then $x^{T}Mx\in
X_{+}$ for every $x\in\left(  X\right)  ^{n}.$ Moreover, in order to prove
that a matrix $M=\left(  m_{i,j}\right)  $ is positive semi-definite it is
enough to check the positivity of $x^{T}Mx$ for elements $x$ in $\left(
\mathcal{R}\left(  T\right)  \right)  ^{n}.$ Indeed if $x\in X_{e}^{n}$ and
$V$ the $f$-subalgebra of $X^{u}$ generated by the set of the coefficients of
$M$ and of $x.$ Then for every $\omega\in H_{m}\left(  V\right)  $ and
$\lambda=\left(  \omega\left(  x_{j}\right)  e,...,\omega\left(  x_{j}\right)
e\right)  ^{T}$ we have
\[
\omega\left(  x^{T}Mx\right)  =\omega\left(  \lambda^{T}M\lambda\right)
\geq0,
\]
as $\lambda\in\mathcal{R}\left(  T\right)  ^{n}$ and $\omega$ is positive.
Thus $M$ is positive semi-definte.
\end{remark}

For a matrix $M$ in $\mathbf{M}_{n}\left(  X\right)  $ let $\Gamma\left(
M\right)  $ denote the sum of all entries of $M,$ that is, $\Gamma\left(
M\right)  =\sum\limits_{1\leq i,j\leq n}m_{ij}.$

\begin{lemma}
\label{YY2-g}Let $\left(  x_{n}\right)  _{n\in\mathbb{N}}$ be a positive
sequence in $X_{+},$ where $X$ is a Dedekind complete Riesz space with weak
order unit $e$, and let $R_{n}=\sum\limits_{k=n}^{\infty}x_{k}\in X_{+}^{s},$
for $n\in\mathbb{N}.$ Then $R_{1}^{\infty}=\left(  \inf R_{n}\right)
^{\infty}.$
\end{lemma}

\begin{proof}
Clearly $R_{1}^{\infty}\geq\left(  \inf R_{n}\right)  ^{\infty}$ and by
\cite[Proposition 21]{L-900},
\[
R_{1}^{\infty}=R_{n}^{\infty}+\left(  \sum\limits_{k=1}^{n-1}x_{k}\right)
^{\infty}=R_{n}^{\infty}.
\]
Assume that $u\in B_{R_{1}^{\infty}}^{+}.$ Then $R_{n}\geq tu$ for all real
$t\geq0.$ So $\inf\limits_{n}R_{n}\geq tu$ for all real $t\geq0.$ This shows
that $u\in B_{\left(  \inf\limits_{n}R_{n}\right)  ^{\infty}}.$ We deduce from
this that $B_{R_{1}^{\infty}}\subseteq B_{\left(  \inf\limits_{n}R_{n}\right)
^{\infty}}$ and so $\left(  \inf\limits_{n}R_{n}\right)  ^{\infty}\geq
R_{1}^{\infty}.$
\end{proof}

We aim to generalize to the setting of Riesz space the following result due to
Feng and Shen \cite{a-1843}.

\begin{theorem}
Let $\left(  \Omega,\mathcal{F},\mathbb{P}\right)  $ be a probability space,
$\left(  A_{n}\right)  $ a sequence of events and $\left(  w_{n}\right)  $ a
sequence of positive reals. If $\sum\limits_{n=1}^{\infty}w_{n}\mathbb{P}%
(A_{n})=\infty$, then%
\[
\mathbb{P}(\lim\sup A_{n})\geq\limsup\limits_{n}\frac{\left(  \sum_{i=1}%
^{n}w_{i}\mathbb{P}(A_{i})\right)  ^{2}}{\sum_{i=1}^{n}\sum_{j=1}^{n}%
w_{i}w_{j}\mathbb{P}(A_{i}A_{j})}.
\]

\end{theorem}

Our inspiration for the proof derives from \cite{a-1843}. But the proof here
is more technical.

For a subalgebra $Y$ of $X^{u}$ we will use $H_{m}\left(  Y\right)  $ to
denote the set of all Riesz and algebra homomorphism from $Y$ to $\mathbb{R}.$
We recall that if $Y$ is generated by a countable set the $H_{m}\left(
Y\right)  $ separates points of $Y$ and we have, in particular, for $y\in Y,$
the equivalence%
\[
y\geq0\Leftrightarrow\varphi\left(  y\right)  \geq0\text{ for all }\varphi\in
H_{m}\left(  V\right)  .
\]

\begin{lemma}
\label{M1}Let $M=\left(  m_{ij}\right)  \in\mathbf{M}_{n}\left(  X^{u}\right)
$ be a positive semi-definite matrix, $V$ a unital subalgebra of $X^{u}$
containing the entries of $M$ and $\varphi:V\longrightarrow\mathbb{R}$ a
positive algebra homomorphism. Then the real matrix $\varphi\left(  M\right)
:=\left(  \varphi\left(  m_{ij}\right)  \right)  $ is positive semi-definite.
\end{lemma}

\begin{proof}
Let $\lambda=\left(  \lambda_{1},...,\lambda_{n}\right)  \in\mathbb{R}^{n}$
and $x=\left(  \lambda_{1}e,\lambda_{2}e,...,\lambda_{n}e\right)  .$ An easy
computation gives $\varphi\left(  xMx^{T}\right)  =\lambda^{T}\varphi\left(
M\right)  \lambda,$ and the result follows.
\end{proof}

\begin{lemma}
\label{M2}Let $n,k_{1},...,k_{n}$ be integers, with $k=k_{1}+...+k_{n}$ and
$A_{i,j}\in\mathbf{M}_{k_{i},k_{j}}\left(  X_{+}\right)  .$

\begin{enumerate}
\item If $M=\left(  A_{ij}\right)  _{1\leq i,j\leq n}\in\mathbf{M}_{k}\left(
X\right)  $ is positive semi-definite then the matrix $S=\left(  \Gamma\left(
A_{ij}\right)  \right)  $ is so.

\item If $M=\left(  m_{ij}\right)  $ is positive semi-definite then $\det
M\geq0.$
\end{enumerate}
\end{lemma}

\begin{proof}
(i) Let $x\in X^{n},$ then $x^{T}Sx=v^{T}Mv\geq0$ with $v=\left(  x_{1}%
,x_{1},...,x_{_{1},}x_{2},...,x_{2},.....,x_{n},...,x_{n}\right)  ^{T}.$

(ii) Let $V$ be the subalgebra of $X^{u}$ generated by the entries of $M$ and
$e.$ For every $\varphi\in H_{m}\left(  V\right)  $ the matrix $\varphi\left(
M\right)  $ is positive semi-definite by Lemma \ref{M1}. So its determinant is
positive. This shows that
\[
\varphi\left(  \det M\right)  =\det\varphi\left(  M\right)  \geq0.
\]
\newline As this holds for every $\varphi\in H_{m}\left(  V\right)  $ we
deduce that $\det M\geq0.$
\end{proof}

\begin{lemma}
\label{M3}Given a partition of an $(m+n)\times(m+n)$ symmetric matrix
$M=(m_{i,j})\in\mathcal{M}_{m+n}(X):$%
\[
M=\left(
\begin{array}
[c]{cc}%
A & C\\
^{t}C & B
\end{array}
\right)  ,
\]
where $A\in\mathbf{M}_{m}(X),$ $B\in\mathbf{M}_{n}(X)$ and $C\in
\mathbf{M}_{m,n}(X).$ If $M$\ is positive semi-definite, then $\Gamma
(C)^{2}\leq\Gamma(A).\Gamma(B).$
\end{lemma}

\begin{proof}
This follows from Lemma \ref{M2} and the fact that $\Gamma(A).\Gamma
(B)-\Gamma(C)^{2}=\det\left(
\begin{array}
[c]{cc}%
\Gamma\left(  A\right)  & \Gamma\left(  C\right) \\
\Gamma\left(  C^{T}\right)  & \Gamma\left(  B\right)
\end{array}
\right)  .$
\end{proof}

\begin{lemma}
\label{M4}Let $\{Q_{i}\}_{i=1}^{n}$ be a sequence of band projections on $X$.
Then the matrix $M=(TQ_{i}Q_{j}e)_{1\leq i,j\leq n}$ is positive semi-definite.
\end{lemma}

\begin{proof}
Let $u=(u_{1},u_{2},...,u_{n})\in\mathcal{R}(T)^{n}.$ Then using the average
property of $T$ we have%
\begin{align*}
uMu^{T}  &  =%
{\textstyle\sum\limits_{i,j}}
u_{i}u_{j}TQ_{i}Q_{j}e=%
{\textstyle\sum\limits_{i,j}}
T\left(  u_{i}u_{j}Q_{i}Q_{j}e\right) \\
&  =T%
{\textstyle\sum\limits_{i,j}}
u_{i}u_{j}Q_{i}Q_{j}e=T\left(
{\textstyle\sum\limits_{i=1}^{n}}
u_{i}Q_{i}e\right)  ^{2}\geq0.
\end{align*}
which proves the desired result.
\end{proof}

Consider now a sequence $(v_{n})$ in $\mathcal{R}(T)_{+}$\ and a sequence
$(Q_{n})$ in $\mathcal{B}(X).$ For any $1\leq q\leq n\leq\infty,$\ let us
define the following

$K_{q,n}:=%
{\textstyle\sum\limits_{i=q}^{n}}
v_{i}TQ_{i}e,$

$R_{q,n}:=%
{\textstyle\sum\limits_{i=q}^{n}}
v_{i}v_{1}TQ_{i}Q_{1}e,$

$S_{q,n}:=%
{\textstyle\sum\limits_{q\leq i,j\leq n}}
v_{i}v_{j}TQ_{i}Q_{j}e,$

$R_{q,n}\left(  j\right)  :=%
{\textstyle\sum\limits_{i=q}^{n}}
v_{j}v_{i}TQ_{i}Q_{j}e,$ for $1\leq j\leq n.$

These notations will be utilized in subsequent discussions, notably in our
main result, Theorem \ref{M7}.

\begin{lemma}
\label{M5}With the aforementioned notations, the following relationship holds true:

\begin{enumerate}
\item[(i)] If $1\leq q\leq n\leq\infty$ then $K_{q,n}^{2}\leq S_{q,n}$.

\item[(ii)] For all $1\leq p\leq q<\infty$ we have%
\[
S_{q,n}^{\ast}S_{p,n}\overset{o}{\longrightarrow}e_{S_{q,\infty}}+S_{q,\infty
}^{\ast}\left(  S_{p,q-1}+2%
{\textstyle\sum_{j=p}^{q-1}}
R_{q,\infty}^{f}\left(  j\right)  \right)  \text{ in }X^{u}\text{ as
}n\longrightarrow\infty
\]
In particular, if the finite part $S_{q,\infty}^{f}$ of $S_{q,\infty}$ is
null, we get
\[
S_{q,n}^{\ast}S_{p,n}\overset{o}{\longrightarrow}e_{S_{q,\infty}\text{ }%
}\text{as }n\longrightarrow\infty.
\]

\end{enumerate}
\end{lemma}

\begin{proof}
(i) Put $x=%
{\textstyle\sum\limits_{i=q}^{n}}
v_{i}Q_{i}e.$ Then by the averaging property of $T$ we have $Tx=%
{\textstyle\sum\limits_{i=q}^{n}}
T\left(  v_{i}Q_{i}e\right)  =%
{\textstyle\sum\limits_{i=q}^{n}}
v_{i}TQ_{i}e.$ Hence it follows from Cauchy-Schwarz Inequality or Lyapunov
inequality \cite{L-180} that%
\[
K_{q,n}^{2}=\left(  Tx\right)  ^{2}\leq Tx^{2}=S_{q,n},
\]

(ii) Write $S_{p,n}=S_{p,q-1}+S_{q,n}+2%
{\textstyle\sum\limits_{j=p}^{q-1}}
R_{q,n}\left(  j\right)  .$ According to Lemma \ref{YY2-q}, we have
$S_{q,n}^{\ast}\overset{o}{\longrightarrow}S_{q,\infty}^{\ast}$ in $X^{u}.$
Moreover, we have clearly%
\begin{equation}
S_{q,n}^{\ast}S_{q,n}=P_{S_{q,n}}e\uparrow P_{S_{q,\infty}}e, \label{A1}%
\end{equation}
and%
\begin{equation}
S_{q,n}^{\ast}S_{p,q-1}\overset{o}{\longrightarrow}S_{q,\infty}^{\ast
}S_{p,q-1}\ \ \text{in }X^{u}. \label{A2}%
\end{equation}
Observe on the other hand that%
\[
R_{q,n}^{2}\left(  j\right)  \leq v_{j}^{2}\left(
{\textstyle\sum\limits_{i=q}^{n}}
v_{i}TQ_{i}e\right)  ^{2}=v_{j}^{2}K_{q,n}^{2}\leq v_{j}^{2}S_{q,n}.
\]
It follows from Lemma \ref{YY2-r} that%
\[
\left(  v_{j}^{2}S_{q,n}\right)  ^{\ast}R_{q,n}\left(  j\right)  \overset
{o}{\longrightarrow}\left(  v_{j}^{\ast}\right)  ^{2}S_{q,\infty}^{\ast
}R_{q,\infty}^{f}\left(  j\right)  \text{ in }X^{u}.
\]
As $S_{q,\infty}^{\ast}R_{q,\infty}^{f}\left(  j\right)  $ belongs to the band
$B_{v_{j}}$ we deduce that%
\[
S_{q,n}^{\ast}R_{q,n}\left(  j\right)  =v_{j}^{2}\left(  v_{j}^{2}%
S_{q,n}\right)  ^{\ast}R_{q,n}\left(  j\right)  \overset{o}{\longrightarrow
}S_{q,\infty}^{\ast}R_{q,\infty}\left(  j\right)  \ \text{in }X^{u}.
\]
As $R_{q,\infty}^{\infty}\left(  j\right)  \leq S_{q,\infty}^{\infty}\perp
S_{q,\infty}^{\ast},$ we have%
\[
S_{q,\infty}^{\ast}R_{q,\infty}\left(  j\right)  =S_{q,\infty}^{\ast
}R_{q,\infty}^{f}\left(  j\right)  .
\]
Hence,%
\begin{equation}
S_{q,n}^{\ast}%
{\textstyle\sum_{j=p}^{q-1}}
R_{q,n}\left(  j\right)  \overset{o}{\longrightarrow}S_{q,\infty}^{\ast}%
{\textstyle\sum_{j=p}^{q-1}}
R_{q,\infty}^{f}\left(  j\right)  \text{ in }X^{u}. \label{A4}%
\end{equation}
The required result follows by combining (\ref{A1}), (\ref{A2}), and (\ref{A4}).
\end{proof}

\begin{lemma}
\label{YY2-i}Let $X$ be a Dedekind complete Riesz space with weak order unit
$e,$ and let $Y$ be a regular Dedekind complete Riesz subspace of $X$ with
$e\in Y.$ Then $x^{f},x^{\infty}\in Y^{s}$ for every element $x\in Y_{+}^{s}.$
In particular this can be applied to $Y=\mathcal{R}\left(  T\right)  ,$ the
range of a conditional expectation operator $T.$
\end{lemma}

\begin{proof}
\label{M8}Let $x$ be an element in $Y_{+}^{s}.$ By \cite[Lemma 12]{L-900} we
have%
\[
x\wedge ne=x^{f}\wedge ne+x^{\infty}\wedge ne\in Y.
\]
It follows that%
\[
x\wedge\left(  n+1\right)  e-x\wedge ne=x^{f}\wedge\left(  n+1\right)
e-x^{f}\wedge ne+P_{x^{\infty}}e\in Y.
\]
Taking the order limit we get $P_{x^{\infty}}e\in Y.$ So
\[
x^{f}\wedge ne=x\wedge ne-ne_{x^{\infty}}\in Y.
\]
Hence%
\[
x^{f}=\sup\limits_{n}\left(  x^{f}\wedge ne\right)  \in Y.
\]
Now we get $x^{\infty}=x-x^{f}\in Y^{s},$ as required.
\end{proof}

\textbf{Remark. }Let $\left(  x_{n}\right)  $ and $\left(  y_{n}\right)  $ be
two bounded sequences in $X_{+}^{u}.$ It is easy to see that if $x_{n}%
\overset{o}{\longrightarrow}x$ then
\[
\limsup\left(  x_{n}y_{n}\right)  =x\limsup y_{n}.
\]
This will be used in the next result.

\begin{proposition}
\label{M6}Using the same notations as previously and let $B$ be the band
generated by $K_{1,\infty}^{\infty}$ and $P$ the corresponding band
projection. Then the following assertions are valid.

\begin{enumerate}
\item[(i)] $R_{1,\infty}^{\infty}\leq K_{1,\infty}^{\infty}\leq S_{1,\infty
}^{\infty}.$

\item[(ii)] The sequences $\left(  K_{q,\infty}^{f}\right)  $ and $\left(
R_{q,\infty}^{f}\right)  $ are decreasing.

\item[(iii)] For all integer $q$ we have $K_{q,\infty}^{\infty}=K_{1,\infty
}^{\infty}$ and $R_{q,\infty}^{\infty}=R_{1,\infty}^{\infty}.$

\item[(iv)] For $V_{n}:=P_{K_{q,n}}S_{1,n}^{\ast}K_{1,n}^{2}$ we have:%
\begin{align*}
\limsup\limits_{n}V_{n}  &  =P_{K_{q,\infty}}\limsup\limits_{n}\left(
S_{1,n}^{\ast}K_{1,n}^{2}\right) \\
&  =\left(  P_{K_{q,\infty}^{f}}S_{q,\infty}^{\ast}S_{1,q-1}+P_{K_{q,\infty}%
}e+2P_{K_{q,\infty}^{f}}S_{q,\infty}^{\ast}%
{\textstyle\sum_{j=1}^{q-1}}
R_{q,\infty}^{f}\left(  j\right)  \right)  ^{\ast}\\
&  \left(  K_{1,q-1}K_{q,\infty}^{\ast}+P_{K_{q,\infty}}e\right)  ^{2}%
\limsup\limits_{n}\left[  S_{q,n}^{\ast}K_{q,n}^{2}\right]  .
\end{align*}

\item[(v)] We have the following equality%
\[
\limsup\limits_{n}\left(  PS_{1,n}^{\ast}K_{1,n}^{2}\right)  =\limsup
\limits_{n}\left(  PS_{q,n}^{\ast}K_{q,n}^{2}\right)  .
\]
In particular, if $B=X,$ then $\limsup\limits_{n}\left(  S_{q,n}^{\ast}%
K_{q,n}^{2}\right)  $ is independent of $q.$
\end{enumerate}
\end{proposition}

The last two points in Lemma \ref{M6} provide a generalization of
\cite[Proposition 6]{a-1843}.

\begin{proof}
(i) By Lemma \ref{M5} $K_{1,n}^{2}\leq S_{1,n}$ and then $K_{1,\infty}^{2}\leq
S_{1,\infty},$ which yields $K_{1,\infty}^{\infty}=\left(  K_{1,\infty}%
^{2}\right)  ^{\infty}\leq S_{1,\infty}^{\infty}.$ For the second inequality
observe that%
\[
R_{1,n}=v_{1}%
{\textstyle\sum\limits_{k=1}^{n}}
v_{k}TQ_{1}Q_{k}e\leq v_{1}K_{1,n},
\]
and we deduce as above that $R_{1,\infty}^{\infty}\leq K_{1,\infty}^{\infty}.$

(ii) This follows from Remark \ref{YY2-o}.

(iii)\textbf{\ }This follows immediately from the equality $\left(
x+y\right)  ^{\infty}=x^{\infty}+y^{\infty}$ which holds for all elements
$x,y\in X_{+}^{s}.$ \cite[Proposition 21]{L-900}.

(iv)\ In view of Lemma \ref{M5}.(i) we have $S_{1,n}^{\ast}.K_{1,n}^{2}\leq
e.$ Hence the first equality follows from the remark preceding Proposition
\ref{M6} and the fact that $P_{K_{q,n}}e\uparrow P_{K_{q,\infty}}e$ as
$n\longrightarrow\infty$. For the second we will use the decomposition%
\[
V_{n}:=P_{K_{q,n}}S_{1,n}^{\ast}K_{1;n}^{2}=P_{K_{q,n}}S_{q,n}^{\ast}%
K_{q,n}^{2}.A_{n}.M_{n}^{\ast},
\]
where%
\[
A_{n}=\left(  K_{1,q-1}K_{q,n}^{\ast}+P_{K_{q,n}}e\right)  ^{2},
\]
and%
\[
M_{n}=P_{K_{q,n}}\left(  S_{1,q-1}S_{q,n}^{\ast}+e+2S_{q,n}^{\ast}%
{\textstyle\sum\limits_{j=1}^{q-1}}
R_{q,n}\left(  j\right)  \right)  .
\]
According to Lemma \ref{M5}, the inequality $K_{q,n}^{2}\leq S_{q,n}$ implies
that $K_{q,n}^{2}S_{q,n}^{\ast}\leq P_{S_{q,n}}e$. Hence the sequence $\left(
K_{q,n}^{2}S_{q,n}^{\ast}\right)  $\ is order bounded in $X^{u}.$ We have%
\[
P_{K_{q,n}}K_{1,n}^{2}=K_{q,n}^{2}\left(  K_{1,q-1}K_{q,n}^{\ast}+P_{K_{q,n}%
}e\right)  ^{2}=K_{q,n}^{2}A_{n},
\]
with%
\[
A_{n}=\left(  K_{1,q-1}K_{q,n}^{\ast}+P_{K_{q,n}}e\right)  ^{2}\overset
{o}{\longrightarrow}\left(  K_{1,q-1}K_{q,\infty}^{\ast}+P_{K_{q,\infty}%
}e\right)  ^{2}\text{ in }X^{u}.
\]
\ Next, we observe that%
\begin{align*}
P_{K_{q,n}}S_{1,n}^{\ast}  &  =P_{K_{q,n}}\left(  S_{1,q-1}+S_{q,n}+2%
{\textstyle\sum\limits_{j=1}^{q-1}}
R_{q,n}\left(  j\right)  \right)  ^{\ast}\\
&  =P_{K_{q,n}}S_{q,n}^{\ast}\left(  S_{q,n}^{\ast}S_{1,q-1}+P_{S_{q,n}%
}e+2S_{q,n}^{\ast}%
{\textstyle\sum\limits_{j=1}^{q-1}}
R_{q,n}\left(  j\right)  \right)  ^{\ast}\\
&  =P_{K_{q,n}}S_{_{q,n}}^{\ast}M_{n}^{\ast}.
\end{align*}
with%
\[
M_{n}=P_{K_{q,n}}\left(  S_{1,q-1}S_{q,n}^{\ast}+e+2S_{q,n}^{\ast}%
{\textstyle\sum\limits_{j=1}^{q-1}}
R_{q,n}\left(  j\right)  \right)  .
\]
Now in the proof of Lemma \ref{M5} it has been shown that
\[
S_{q,n}^{\ast}R_{q,n}\left(  j\right)  \overset{o}{\longrightarrow}%
S_{q,\infty}^{\ast}R_{q,\infty}\left(  j\right)  \text{ in }X^{u}.
\]
So we have%
\begin{align*}
M_{n}\overset{o}{\longrightarrow}M  &  :=P_{K_{q,\infty}}\left(
S_{1,q-1}S_{q,\infty}^{\ast}+e+2S_{q,\infty}^{\ast}%
{\textstyle\sum\limits_{j=1}^{q-1}}
R_{q,\infty}\left(  j\right)  \right)  \text{ in }X^{u},\\
&  =P_{K_{q,\infty}^{f}}S_{1,q-1}S_{q,\infty}^{\ast}+P_{K_{q,\infty}%
}e+2P_{K_{q,\infty}^{f}}S_{q,\infty}^{\ast}%
{\textstyle\sum\limits_{j=1}^{q-1}}
R_{q,\infty}^{f}\left(  j\right)  \text{ }%
\end{align*}
where we use in the last equality the fact that $R_{q,\infty}^{\infty}\left(
j\right)  \perp S_{q,\infty}^{\ast}$ ($1\leq j\leq q-1$) and that
$K_{q,\infty}^{\infty}\perp S_{q,\infty}^{\ast}$. This shows, in particular,
that $\left(  M_{n}\right)  $ is order bounded in $X^{u}.$ On the other hand
we have by definition of $M_{n}$ that $M_{n},M_{n}^{\ast}\in B_{K_{q,n}%
}\subset B_{S_{q,n}}$ and%
\[
S_{q,n}^{\ast}M_{n}^{\ast}=P_{K_{q,n}}S_{q,n}^{\ast}M_{n}^{\ast}=P_{K_{q,n}%
}S_{1,n}^{\ast}\leq S_{1,n}^{\ast}.
\]
Hence%
\[
M_{n}^{\ast}\leq S_{q,n}S_{1,n}^{\ast}=S_{1,n}S_{1,n}^{\ast}\leq e.
\]
This shows that $\left(  M_{n}^{\ast}\right)  $ is order bounded in $X^{u}.$
Now using Lemma \ref{YY3-a} we get%
\[
P_{M}e=P_{\liminf M_{n}}e\leq\liminf P_{M_{n}}e=\liminf\left(  M_{n}%
M_{n}^{\ast}\right)  =M\liminf M_{n}^{\ast}.
\]
Thus $P_{M}e\leq M\liminf M_{n}^{\ast}$ and then%
\begin{equation}
M^{\ast}\leq\liminf M_{n}^{\ast}. \label{M6-1}%
\end{equation}
Similarly we have%
\[
M\limsup M_{n}^{\ast}=\limsup\left(  M_{n}M_{n}^{\ast}\right)  =\limsup
e_{M_{n}}\leq e.
\]
Thus $P_{M}\limsup M_{n}^{\ast}\leq M^{\ast}.$ Now as $\left(  M_{n}^{\ast
}\right)  $ is contained in the band $B_{K_{q,\infty}}$ and $M$ is a weak
order unit of that band; the last inequality becomes%
\begin{equation}
\limsup M_{n}^{\ast}\leq M^{\ast}. \label{M6-2}%
\end{equation}
We deduce from (\ref{M6-1}) and (\ref{M6-2}) that $M_{n}^{\ast}\overset
{o}{\longrightarrow}M^{\ast}$\ in $X^{u}.$ Observe finally that each of the
three sequences $\left(  e_{K_{q,n}}S_{q,n}^{\ast}K_{q,n}^{2}\right)  ,$
$\left(  A_{n}\right)  $ and $\left(  M_{n}\right)  $ is bounded in $X^{u},$
which enables us to use Lemma \ref{YY3-a} and conclude that%
\begin{align*}
\limsup\limits_{n}\left(  P_{K_{q,n}}S_{1,n}^{\ast}K_{1,n}^{2}\right)   &
=\limsup\limits_{n}\left[  \left(  K_{q,n}^{2}S_{q,n}^{\ast}\right)
M_{n}^{\ast}A_{n}\right] \\
&  =\limsup\limits_{n}\left[  K_{q,n}^{2}S_{q,n}^{\ast}\right]  .\lim
M_{n}^{\ast}.\lim\limits_{n}A_{n}\\
&  =\limsup\limits_{n}\left[  K_{q,n}^{2}S_{q,n}^{\ast}\right]  .M^{\ast
}\left(  K_{1,q-1}\left(  K_{q,\infty}^{f}\right)  ^{\ast}+e_{K_{q,\infty}%
}\right)  ^{2}%
\end{align*}
in $X^{u}.$ This completes the proof of (iv).

(v) This is an immediate consequence of (iv) by applying $P$ to (iv) and using
the equality $K_{1,\infty}^{\infty}=K_{q,\infty}^{\infty}$ proved in (iii).
\end{proof}

We reach now the central result of this section.

\begin{theorem}
\label{M7}Let $\left(  X,e,T\right)  $ be a conditional Riesz
triple,\textbf{\ }$\left(  v_{n}\right)  $ a sequence in $\mathcal{R}(T)_{+}$
and $\left(  Q_{n}\right)  $ a sequence of band projections. Let $P$ be the
projection on the band generated by the infinite part of $%
{\textstyle\sum\limits_{i=1}^{\infty}}
v_{i}TQ_{i}e.$ Then ($Q_{v_{n}}=P_{v_{n}}Q_{n}e.$)%
\[
TP\limsup\limits_{n}Q_{v_{n}}e\geq\limsup\limits_{n}\left(  P\left(
S_{1,n}^{\ast}K_{1,n}^{2}\right)  \right)  .
\]
In particular, if $\left(
{\textstyle\sum\limits_{i=1}^{\infty}}
v_{i}TQ_{i}e\right)  ^{f}=0$ then%
\[
T\left(  \limsup\limits_{n}Q_{v_{n}}e\right)  \geq\limsup\limits_{n}\left(
S_{1,n}^{\ast}.K_{1,n}^{2}\right)  .
\]

\end{theorem}

\begin{proof}
(i) Let $B$ be the band associated to $P.$ By Lemma \ref{YY2-i} $\infty_{B}%
\in\mathcal{R}\left(  T\right)  ^{s}$ and by \cite[Proposition 7]{L-900} we
get $P_{B}T=TP_{B}.$ Put $Q_{v_{i}}=Q_{i}P_{v_{i}}.$ Then $c:=%
{\textstyle\bigvee\limits_{i=q}^{n}}
Q_{v_{i}}e$ is a component of $e$ and it follows from Cauchy-Schwarz
inequality that%
\[
K_{q,n}^{2}=\left(  T\left(  c.%
{\textstyle\sum\limits_{i=q}^{n}}
v_{i}Q_{i}e\right)  \right)  ^{2}\leq Tc.S_{q,n}.
\]
So as $T$ commutes with $P,$ we obtain%
\[
PK_{q,n}^{2}\leq T\left(  Pc\right)  .P\left(  S_{q,n}\right)  =T%
{\textstyle\bigvee\limits_{i=q}^{n}}
PQ_{v_{n}}e.P\left(  S_{q,n}\right)  ,
\]
which implies that%
\begin{align*}
P\left(  K_{q,n}^{2}S_{q,n}^{\ast}\right)   &  =PK_{q,n}^{2}.PS_{q,n}^{\ast
}\leq T\left(
{\textstyle\bigvee\limits_{i=q}^{n}}
PQ_{v_{n}}e\right)  .PS_{q,n}.PS_{q,n}^{\ast}\\
&  =T\left(
{\textstyle\bigvee\limits_{i=q}^{n}}
PQ_{v_{n}}e\right)  ,
\end{align*}
\newline where the last equality holds because $T\left(
{\textstyle\bigvee\limits_{i=q}^{n}}
PQ_{v_{n}}e\right)  $ belongs to the band $B\cap B_{S_{q,n}}$ and
$PS_{q,n}.PS_{q,n}^{\ast}$ is the component of $e$ on this band. It follows
that%
\begin{align*}
a  &  :=T\left(  \limsup\limits_{n}PQ_{v_{n}}e\right)  =\lim_{q}\sup_{n\geq
q}T\left(
{\textstyle\bigvee\limits_{i=q}^{n}}
PQ_{v_{i}}e\right) \\
&  \geq\lim_{q}\limsup\limits_{n}P\left(  S_{q,n}^{\ast}.K_{q,n}^{2}\right)
=\limsup\limits_{n}P\left(  S_{1,n}^{\ast}.K_{1,n}^{2}\right)  ,
\end{align*}
which proves the desired inequality. Here the last equality follows from Lemma
\ref{M5}(v). In the case $\left(
{\textstyle\sum\limits_{i=1}^{\infty}}
v_{i}TQ_{i}e\right)  ^{f}=0$ we get\ $K_{1,n}\in B$ for every $n,$ and then
$P\left(  S_{1,n}^{\ast}K_{1,n}^{2}\right)  =S_{1,n}^{\ast}.K_{1,n}^{2}.$ It
follows from the first case that%
\[
T\left(  \limsup\limits_{n}Q_{v_{n}}e\right)  \geq\limsup\limits_{n}\left(
S_{1,n}^{\ast}.K_{1,n}^{2}\right)  .
\]
\newline This completes the proof.
\end{proof}

Applying Theorem \ref{M7} to $v_{n}=(Tq_{n})^{\ast}$ with $q_{i}=Q_{i}e,$ we
obtain the following result which is a generalization of \cite[Corollary
2]{a-1843}.

\begin{corollary}
\label{M10}Under the hypothesis of Theorem \ref{M7} with $q_{i}=Q_{i}e$\ we
have%
\[
T\left(  \limsup\limits_{n}P_{Tq_{n}}q_{n}\right)  \geq\limsup\limits_{n}%
\left(
{\textstyle\sum\limits_{1\leq i,j\leq n}}
\left(  Tq_{i}Tq_{j}\right)  ^{\ast}T\left(  q_{i}q_{j}\right)  \right)
^{\ast}.\left(
{\textstyle\sum\limits_{1\leq i\leq n}}
e_{Tq_{i}}\right)  ^{2}.
\]
Theorem \ref{M7} allows to get a generalization of Borel-Cantelli Lemma proved
by the author \cite[Theorem 31]{L-900}.
\end{corollary}

\begin{corollary}
Let $\left(  P_{n}\right)  _{n\geq1}$ be a sequence of parwise $T$-independent
band projections on $X.$ If $B$ is a band on $X^{u}$ such that%
\[%
{\textstyle\sum\limits_{n=1}^{\infty}}
TP_{n}e=\infty_{B}+u\qquad\text{with }u\in B^{d},
\]
then $P_{B}$ commutes with $T$ and%
\[
P_{B}=\limsup\limits_{n}P_{n}.
\]

\end{corollary}

\begin{proof}
We know that $P_{B}$ commutes with $T$ (see the proof of Theorem \ref{M7}).
Now observe that
\[%
{\textstyle\sum\limits_{n=1}^{\infty}}
TP_{B}^{d}P_{n}e=P_{B}^{d}%
{\textstyle\sum\limits_{n=1}^{\infty}}
TP_{n}e=u\in X^{u}.
\]
Hence, accorfing to \cite[Lemma 26]{L-900}, $P_{B}^{d}\limsup Q_{n}=0,$ which
implies the inequality%
\[
\limsup P_{n}\leq P_{B}.
\]
\ \ \ The reverse inequality will follow from Theorem \ref{M7} Indeed if we
apply the theorem with $v_{n}=e$ we get%
\[
T\limsup\limits_{n}Q_{n}=TP_{B}\limsup\limits_{n}Q_{n}e\geq\limsup
\limits_{n}\left(  P_{B}\left(  S_{1,n}^{\ast}K_{1,n}^{2}\right)  \right)
\text{ \ \ \ \ \textbf{(A1)}}%
\]
Moreover, we have by $T$-independance,%
\[
S_{1,n}\leq K_{1,n}^{2}+K_{1,n},
\]
which yields%
\[
P_{S_{1,n}}=S_{1,n}^{\ast}S_{1,n}\leq S_{1,n}^{\ast}K_{1,n}^{2}+S_{1,n}^{\ast
}K_{1,n},
\]
and then
\[
P_{B}P_{S_{1,n}}\leq P_{B}S_{1,n}^{\ast}K_{1,n}^{2}+P_{B}S_{1,n}^{\ast}%
K_{1,n}.
\]
Taking the $\limsup$ over $n$ in this last inequality gives%
\[
P_{B}P_{S_{1,\infty}}\leq\limsup\left(  P_{B}S_{1,n}^{\ast}K_{1,n}^{2}\right)
+\limsup\left(  P_{B}S_{1,n}^{\ast}K_{1,n}\right)  .
\]
Since $\infty_{B}=K_{1,\infty}^{\infty}\leq S_{1,\infty}^{\infty}$ we have
$P_{B}P_{S_{1,\infty}}=P_{B}P_{S_{1,\infty}^{\infty}}=P_{B}.$ So%
\begin{equation}
P_{B}P_{S_{1,\infty}}=P_{B}P_{S_{1,\infty}^{f}}+P_{B}P_{S_{1,\infty}^{\infty}%
}=P_{B}P_{S_{1,\infty}^{\infty}}=P_{B}. \label{C2}%
\end{equation}
Now, $K_{1,n}$ and $S_{1,n}$ are increasing with $K_{1,n}^{2}\leq S_{1,n},$ we
obtain thanks to Proposition \ref{YY2-r},%
\begin{equation}
S_{1,n}^{\ast}K_{1,n}\overset{o}{\longrightarrow}S_{1,\infty}^{\ast
}K_{1,\infty}^{f}, \label{C3}%
\end{equation}
which gives%
\[
P_{B}S_{1,n}^{\ast}K_{1,n}\overset{o}{\longrightarrow}P_{B}S_{1,\infty}^{\ast
}K_{1,\infty}^{f}=0.
\]
Combining (\ref{C2}) and (\ref{C3}) we derive that%
\begin{align*}
TP_{B}e  &  =P_{B}e=P_{B}P_{S_{1,\infty}}e\leq\limsup\left(  P_{B}%
S_{1,n}^{\ast}K_{1,n}^{2}\right) \\
&  \leq TP_{B}\limsup\limits_{n}P_{n}e.
\end{align*}
Applying Theorem \ref{M7} gives%
\[
TP_{B}e\leq T\limsup\limits_{n}P_{n}e.
\]
Therefore, since $T$ is strictly positive we conclude that%
\[
P_{B}e=\limsup\limits_{n}P_{n}e.
\]
his complete the proof.
\end{proof}

\end{document}